\documentclass[final,onefignum,onetabnum]{siamart190516}

\usepackage{amsmath,amsfonts,amssymb,mathtools,nicefrac,mymath}
\usepackage{graphicx,rotating}
\usepackage[caption=false]{subfig}
\usepackage[algo2e,ruled]{algorithm2e}

\definecolor{myblue}{rgb}{0,0.4470,0.7410}
\definecolor{myred}{rgb}{0.8500,0.3250,0.0980}
\definecolor{myorange}{rgb}{0.9290,0.6940,0.1250}
\definecolor{mypurple}{rgb}{0.4940,0.1840,0.5560}
\definecolor{mygreen}{rgb}{0.4660,0.6740,0.1880}
\definecolor{mylightblue}{rgb}{0.3010,0.7450,0.9330}
\definecolor{mydarkred}{rgb}{0.6350,0.0780,0.1840}

\usepackage{tikz}
\usepackage{pgfplots}
\usetikzlibrary{calc}
\usetikzlibrary{positioning}
\pgfplotsset{
  compat=newest,
  table/header=false,
  tick label style={font=\scriptsize},
  label style={font=\scriptsize},
  legend style={font=\scriptsize},
  legend cell align=left,
  colormap={parula}{
    rgb255=(53,42,135)
    rgb255=(15,92,221)
    rgb255=(18,125,216)
    rgb255=(7,156,207)
    rgb255=(21,177,180)
    rgb255=(89,189,140)
    rgb255=(165,190,107)
    rgb255=(225,185,82)
    rgb255=(252,206,46)
    rgb255=(249,251,14)
  }
}
\pgfplotsset{
  myColOne/.style={myblue},
  myColTwo/.style={myred},
  myColThr/.style={myorange},
  myColFou/.style={mypurple},
  myColFiv/.style={mygreen},
  myColSix/.style={mylightblue},
  myColSev/.style={mydarkred}
}
\pgfplotsset{
  myStyOne/.style={myColOne,thick,mark=+},
  myStyTwo/.style={myColTwo,thick,mark=+},
  myStyThr/.style={myColThr,thick,mark=+},
  myStyFou/.style={myColFou,thick,mark=+},
  myStyFiv/.style={myColFiv,thick,mark=+},
  myStySix/.style={myColSix,thick,mark=+},
  myStySev/.style={myColSev,thick,mark=+},
  myStyRes/.style={black,densely dotted},
}
\pgfplotsset{
  myStyOn/.style={myblue,     thick,mark=asterisk},
  myStyTw/.style={myred,      thick,mark=asterisk},
  myStyTh/.style={myorange,   thick,mark=asterisk},
  myStyFo/.style={mypurple,   thick,mark=asterisk},
  myStyFi/.style={mygreen,    thick,mark=asterisk},
  myStySi/.style={mylightblue,thick,mark=asterisk},
  myStySe/.style={mydarkred,  thick,mark=asterisk},
}
\pgfplotsset{
  every axis/.append style={
    label style={font=\footnotesize},
  }
}
\newcommand{\figname}[1]{}
\newcommand{\fignames}[1]{}
\newcommand{\datfile}[1]{dat/#1.dat}

\usepackage{mytensor}
\usepackage{myfigures}

\usepackage[binary-units]{siunitx}
\usepackage{tabularx,multirow,booktabs}
\newcolumntype{L}{>{\raggedright\arraybackslash}X}
\newcolumntype{C}{>{\centering\arraybackslash}X}
\newcolumntype{R}{>{\raggedleft\arraybackslash}X}

\crefname{subsection}{section}{sections}
\Crefname{subsection}{Section}{Sections}
\newsiamthm{assumption}{Assumption}
\crefname{assumption}{Assumption}{Assumptions}
\newsiamthm{conjecture}{Conjecture}
\crefname{conjecture}{Conjecture}{Conjectures}
\newsiamremark{example}{Example}
\crefname{example}{Example}{Examples}
\newsiamremark{remark}{Remark}
\crefname{remark}{Remark}{Remarks}
\crefname{algocf}{Algorithm}{Algorithms}

\newcommand{\TheTitle}%
{Solving a class of infinite-dimensional tensor eigenvalue problems by translational invariant tensor ring approximations}
\newcommand{\TheAuthors}%
{R.~Van~Beeumen, L.~Peri\v{s}a,  D.~Kressner, and C.~Yang}

\headers{Flexible Power Method for infinite TEP}{\TheAuthors}

\title{{\TheTitle}\thanks{Submitted to the editors \today.
\funding{This work was partially funded by the U.S.~Department of Energy,
Office of Science, Office of Advanced Scientific Computing Research,
Scientific Discovery through Advanced Computing (SciDAC) program
under Contract No. DE-AC0205CH11231.}}}

\author{
    Roel~Van~Beeumen%
    \thanks{Applied Mathematics and Computational Research Division,
            Lawrence Berkeley National Laboratory,
            Berkeley, CA 94720, United States.
            (\email{rvanbeeumen@lbl.gov}, \email{cyang@lbl.gov}).}
  \and
    Lana~Peri\v{s}a%
    \thanks{Visage Technologies, Diskettgatan 11A,
            SE-583 35 Link\"{o}ping, Sweden.
            (\email{lana.perisa@visagetechnologies.com}).}
  \and
    Daniel~Kressner%
    \thanks{Department of Mathematics,
            \'Ecole Polytechnique F\'ed\'erale de Lausanne,
            Station 8, 1015 Lausanne, Switzerland.
            (\email{daniel.kressner@epfl.ch}).}
  \and
    Chao~Yang%
    \footnotemark[2]
}

\usepackage{amsopn}

\newcommand{\can}{\mathrm{c}}
\newcommand{\odd}{\mathrm{o}}
\newcommand{\even}{\mathrm{e}}
\newcommand{\TQL}{T_{Q_{\!_L}}\!}
\newcommand{\TQR}{T_{Q_{\!_R}}\!}
\newcommand{\TUL}{T_{U_{\!_L}}\!}
\newcommand{\TUR}{T_{U_{\!_R}}\!}
\newcommand{\TQUL}{T_{Q_{\!_L}\!U_{\!_L}}\!}
\newcommand{\TQUR}{T_{Q_{\!_R}\!U_{\!_R}}\!}
\newcommand{\TUQL}{T_{U_{\!_L}\!Q_{\!_L}}\!}
\newcommand{\TUQR}{T_{U_{\!_R}\!Q_{\!_R}}\!}
\newcommand{\TtQL}{\Tt_{Q_{\!_L}}\!}
\newcommand{\TtQR}{\Tt_{Q_{\!_R}}\!}
\newcommand{\TtQUL}{\Tt_{Q_{\!_L}\!U_{\!_L}}\!}
\newcommand{\TtQUR}{\Tt_{Q_{\!_R}\!U_{\!_R}}\!}
\newcommand{\TtUQL}{\Tt_{U_{\!_L}\!Q_{\!_L}}\!}
\newcommand{\TtUQR}{\Tt_{U_{\!_R}\!Q_{\!_R}}\!}

\ifpdf
\hypersetup{
  pdftitle={\TheTitle},
  pdfauthor={\TheAuthors}
}
\fi

\begin{document}

\maketitle

\begin{abstract}
We examine a method for solving an infinite-dimensional tensor eigenvalue
problem $H x = \lambda x$, where the infinite-dimensional symmetric matrix
$H$ exhibits a translational invariant structure.
We provide a formulation of this type of problem from a numerical linear algebra point of view and describe how a power method applied to 
$e^{-Ht}$ is used to obtain an approximation to the desired eigenvector.
This infinite-dimensional eigenvector is represented in a compact way by a
translational invariant infinite Tensor Ring (iTR).
Low rank approximation is used to keep the cost of subsequent power iterations
bounded while preserving the iTR structure of the approximate eigenvector.
We show how the averaged Rayleigh quotient of an iTR eigenvector approximation can be efficiently computed and introduce a projected residual to monitor its convergence.
In the numerical examples, we illustrate that the norm of this projected iTR residual
can also be used to automatically modify the time step $t$ to ensure accurate
and rapid convergence of the power method.
\end{abstract}



\section{Introduction}
\label{sec:intro}

To motivate the infinite-dimensional eigenvalue problems considered in this work, we will first consider the finite-dimensional eigenvalue problem
$\cH_\ell x_\ell = \lambda\p\ell x_\ell$,
where the symmetric matrix $\cH_\ell \inRR{d^{2\ell+1}}$
is defined as the following sum of Kronecker products
\begin{align}
\cH_\ell &= \sum_{k=-\ell}^{\ell-1} \Ht_k, &
   \Ht_k &= I \otimes \cdots \otimes I \otimes M_{k,k+1}
              \otimes I \otimes \cdots \otimes I,
\label{eq:Hell}
\end{align}
with the $d \times d$ identity matrix $\eye$ and $M_{k,k+1} \inRR{d^2}$.
We will assume that the matrices $M_{k,k+1}$ are the same for every $k$;
the subscript $k,k+1$ is used to merely
indicate the overlapping positions of $M$ in each Kronecker product.
Note that $x_\ell$ can equivalently be viewed as a $d\times \cdots \times d$ tensor of order $2\ell + 1$. This type of tensor eigenvalue problems originates from the study of model quantum many-body systems such as a quantum spin chain with nearest neighbor interactions~\cite{bobl1981}.
The Kronecker structure of $\cH_\ell$ allows for efficiently
representing such matrices and their eigenvectors in, e.g., the Tensor Train (TT) format~\cite{osel2011}.
The TT format of an eigenvector takes the form
\begin{equation} \label{eq:tt}
 x_\ell(i_{-\ell},i_{-\ell+1},\ldots,i_{\ell-1},i_\ell) = 
X_{-\ell}(i_{-\ell}) X_{-\ell+1}(i_{-\ell+1}) \cdots X_{\ell-1}(i_{\ell-1}) X_\ell(i_\ell),
\end{equation}
where $i_k = 1,\ldots,d$, with $k = -\ell,\ldots,\ell$, and each $X_k(i_k)$ is an $r_k \times r_{k+1}$ matrix, with $r_{-\ell} = r_{\ell+1} = 1$.

In this paper, we consider $\cH_\ell$ as $\ell \to \infty$.
Such limits are known in the physics literature as the 
\emph{thermodynamic limit} and they are important for describing macroscopic
properties of quantum materials~\cite{roos1999}.
Taking $\ell \to \infty$ in~\eqref{eq:Hell} \emph{formally} leads to the following infinite-dimensional symmetric matrix
\begin{align}
\bH &= \sum_{k=-\infty}^{+\infty} \bH_k, &
\bH_k &= \cdots \otimes I \otimes I \otimes M_{k,k+1}
                \otimes I \otimes I \otimes \cdots,
\label{eq:H}
\end{align}
where, again, all matrices $M_{k,k+1}$ are identical and, hence, $\bH$ is called \emph{translational invariant}.
It is important to emphasize that we make no attempt to mathematically frame or justify convergence of the formal series~\eqref{eq:H}. In fact, even the eigenvalues of~\eqref{eq:H} are usually infinite in the sense that the eigenvalues of the matrix $\cH_\ell$ defined in~\eqref{eq:Hell} become unbounded as $\ell$ grows.
To avoid this, one considers averaged eigenvalues of the form $\lambda\p\ell/(2\ell)$, where $\lambda\p\ell$ is an eigenvalue of $\cH_\ell$.
We are particularly interested in the smallest among these averaged eigenvalues as $\ell \to \infty$. This 
is referred to in the physics literature as the ground state energy per
site~\cite{aina1994,yaya1966}.
For some specific choices of $M$, the analytical solutions of these
eigenvalue problems are known~\cite{beth1931,hult1938,batc2007}.
However, in general, the desired averaged eigenvalue and its corresponding
eigenvector need to be computed numerically.
One way to study such an infinite problem computationally is to start with a
finite value of $\ell$ and examine how the smallest averaged eigenvalue
changes as $\ell$ increases.
However, applying a standard eigensolver to $\cH_\ell$ severely limits the feasible range of $\ell$ even on a powerful supercomputer, and one may not attain a reasonably good approximation to the limit $\ell \to \infty$.
It is therefore preferable to directly work with the $\bH$ from~\eqref{eq:H}.
Due to the infinite dimensionality, special care
must be exercised when computing quantities such as the Rayleigh quotient and
the residual of an approximate eigenpair of $\bH$.

As the eigenvectors of $\bH$ are infinite-dimensional vectors, they obviously
cannot be computed or stored directly in memory.
On the other hand, we can represent them in a compact form by
making use of the translational invariance property of $\bH$.
Such an invariance property was studied by Bethe~\cite{beth1931} and
Hulth\'{e}n~\cite{hult1938}, known as the \emph{Bethe--Hulth\'{e}n hypothesis}, which
imposes that the elements of the eigenvector are invariant with respect
to a cyclic permutation of the tensor indices, i.e.,
\(
\bfx(\ldots,i_{-1},i_0,i_1,\ldots) = \bfx(\ldots,i_{0},i_1,i_2,\ldots),
\)
where $i_k = 1,\ldots,d$, and the sequence of indices
$\ldots,i_{-1},i_0,i_1,\ldots$ specifies a particular element of the
infinite-dimensional vector $\bfx$.
We (approximately) represent the eigenvector to be computed as a
translational invariant \emph{infinite Tensor Ring} (iTR), defined as
\[
\bfx(\ldots,i_{-1},i_{0},i_1,\ldots)
 = \trace \left[ \prod_{k=-\infty}^{+\infty} X(i_k) \right],
\]
where $i_k = 1,\ldots,d$ for every $k$ and each $X(i_k)$ is an $r \times r$ matrix.
Note that, thanks to the translational invariance, this allows us to store 
and work with $d$ matrices of size $r \times r$, which is manageable
as long as the rank $r$ is not too large.
An iTR can be seen as the infinite limit of a finite size tensor
ring~\cite{zhzh2016} and is also known as a uniform Matrix Product State
(uMPS)~\cite{zava2018}.

We assume that the averaged eigenvalue of interest is simple and propose to apply
a power iteration to $e^{-\bH t}$ for some small and adjustable
parameter $t>0$ to compute the desired eigenpair.
We should note that the power method is not the only method
for solving this type of infinite-dimensional tensor eigenvalue problems.
Alternative methods include the infinite density matrix renormalization (iDMRG)
method \cite{whit1992} and the variational uniform matrix product states (vUMPS) method \cite{zava2018}.
The main purpose of our paper is to formulate this type of problem from
a numerical linear algebra point of view and describe how a standard numerical
algorithm such as the power method can be modified and applied to address
such problems.

In order to implement the power method, we must be able to multiply a
matrix exponential with an iTR and keep the product in iTR form.
Computing $e^{-\bH t} \bfx$, with $\bH$ being an infinite-dimensional matrix and
$\bfx$ an iTR, is in general not possible.
However, the special structure of \eqref{eq:H} allows us to split $\bH$ in
even and odd terms, $\bH_\even$ and $\bH_\odd$, respectively.
For sufficiently small $t$, we can use the Lie product
formula, also known as Suzuki--Trotter splitting,
to approximate $e^{-\bH t} \bfx$ by only local tensor contractions with
a $d^2 \times d^2$ matrix exponential $e^{-Mt}$.
Such contractions can be implemented in a way that the translational invariance
of the contracted tensor is preserved, although the rank generally increases.
To keep the cost of subsequent power iterations bounded, a low-rank
approximation through the use of a truncated singular value decomposition
is applied.

We should note that the approach described above was first developed in the physics community and called
the infinite time evolution block decimation (iTEBD) method~\cite{vida2007}. We
prefer to refer to this approach as a \emph{flexible} power method to clearly articulate the 
mathematical and algorithmic features of this numerical method. The use of the term ``flexible'' highlights the ability of the method to work with approximations of the matrix exponential $e^{-\bH t}$.
In the physics literature, the accuracy of the approximate eigenvalue is assessed by either comparing the approximation with the exact solution or examining the difference between the approximate eigenvalues obtained from two consecutive iterations.  The former is not practical when the exact solution of the problem is unknown. The latter can be misleading when the power iteration starts to stagnate, which does happen as shown in our numerical experiments.
Instead, we will show that the norm of a suitably defined residual can be
used to monitor convergence.
We give a practical procedure for adjusting the parameter $t$ in the flexible
power iteration to ensure rapid and stable convergence and provide a stopping
criterion for terminating the power iteration.
As illustrated in the numerical experiments, such a procedure is effective.

The rest of this paper is organized as follows.
In \secref{sec:not}, we introduce diagrammatic notations for scalars, vectors,
matrices, and tensors, as well as for tensor operations. These notations make it easier to
explain operations performed on tensor rings.
In \secref{sec:itr}, we begin with a formal definition of a single core
translational invariant infinite Tensor Ring, describe its properties,
and show how an iTR can be converted into a so-called canonical form.
We also extend these definitions and properties to 2-core translational
invariant iTRs which are used to represent the approximate eigenvector of $\bH$ 
containing nearest neighbor interactions as in~\eqref{eq:H}.
In \secref{sec:eigval-approx}, we define suitable notions of Rayleigh quotient and residual, and show how they
can be evaluated efficiently with respect to an iTR.
In \secref{sec:fpm}, we present the power iteration for computing
the smallest eigenpair of $\bH$.
We also show how $e^{\bH t} \bfx$ is computed to preserve the iTR structure
and discuss a number of practical computational considerations regarding the
convergence and error assessment of the method.
Three numerical examples are given in \secref{sec:exp} to demonstrate the
effectiveness of the power method and the adaptive strategy
proposed in \secref{sec:fpm}.
Finally, the main conclusions are summarized in \secref{sec:concl}.

\section{Notation}
\label{sec:not}

Throughout the paper, we denote scalars by lower case Greek characters, vectors
by lower case Roman characters, and matrices and higher order tensors by upper case 
Roman characters, e.g., $\alpha$, $x$, and $M$, respectively.
$\eye[d]$ is the $d \times d$ identity matrix and diagonal matrices are
denoted by upper case Greek characters, e.g., $\Sigma$.
The trace of a matrix $M$ is given by $\trace(M)$ and its vectorization
by $\vecz(M)$.
For infinite-dimensional vectors and matrices we use boldface Roman
characters, e.g., $\bfv$ and $\bM$, respectively.

We will make use of the powerful tensor diagram notation~\cite{penr1971}
 to graphically represent tensor contractions.
The first 4 low-order tensors, i.e., a scalar~$\alpha$, a vector~$v$,
a matrix~$M$, and a 3rd-order tensor~$T$, are depicted as follows
\begin{align*}
\alpha &= \figname{notation-sca}\mytikznotsca{$\alpha$}\,,&
     v &= \figname{notation-vec}\mytikznotvec{$v$}{$i$}\,,&
     M &= \figname{notation-mat}\mytikznotmat{$M$}{$i$}{$j$}\,,&
     T &= \figname{notation-ten}\mytikznotten{$T$}{$i$}{$k$}{$j$}\,,
\end{align*}
where $i,j,k$ are the corresponding tensor indices.
In this graphical notation, tensor contractions can be
represented by connecting the matching lines that correspond to the indices to sum over.
For example, the matrix-vector product, the matrix-matrix product, and the
contraction of two 3rd-order tensors into a so called \emph{supercore}
are given respectively by the following diagrams
\begin{align*}
Mv &= \figname{notation-matvec}\mytikznotmatvec{$M$}{$v$}{$i$}{$j$}\,,&
MN &= \figname{notation-matmat}\mytikznotmatmat{$M$}{$N$}{$i$}{$j$}{$k$}\,,&
\figname{notation-tenten-super}\mytikznottentensuper{$Z$}{$i$}{$kl$}{$m$} &=
\figname{notation-tenten}\mytikznottenten{$X$}{$Y$}{$i$}{$k$}{$j$}{$l$}{$m$}\,,
\end{align*}
where the connections between $M$ and $v$, $M$ and $N$, and $X$ and $Y$
correspond to the sum over index $j$.
Note that the thick leg indicates combining the indices $k$ and $l$ into a
single index.
Because contracting a tensor over one of its indices with the identity matrix
has no effect, we represent the identity matrix $\eye$ by just a line.

Some other commonly used matrix operations such as the trace, the transpose, and
the vectorizations of a matrix are given by the following diagrams
\begin{align*}
 \trace(M) &= \,\figname{notation-trace}\mytikznottrace{$M$}\,,&
       v\T &= \,\figname{notation-vectrans}\mytikznotvectrans{$v$}\,,&
  \vecz(M) &=   \figname{notation-vecM}\mytikznotvecM{$M$}\,,&
\vecz(M)\T &= \,\figname{notation-vecMtrans}\mytikznotvecMtrans{$M$}\,.
\end{align*}
In order to save space, we sometimes rotate the diagrams counterclockwise by 90
degrees.
Throughout the paper, reshaping tensors into matrices and vice versa will play a
key role, e.g., reshaping the following rank-1 matrices into 4th-order
tensors
\begin{align*}
\vecz(\eye) \vecz(\Sigma^2)\T &=
  \figname{notation-vecIvecS2T}\mytikznotvecIvecST{$\Sigma$}\,, &
\vecz(\Sigma^2) \vecz(\eye)\T &=
  \figname{notation-vecS2vecIT}\mytikznotvecSvecIT{$\Sigma$}\,,
\end{align*}
where $\Sigma$ is a diagonal matrix.
We again obtain a matrix by combining respectively the left and right pointing
legs into one index.

\section{Translational invariant infinite tensor ring}
\label{sec:itr}

In this section, we define a translational invariant infinite tensor ring
and discuss its properties that will be used in the next sections to
construct an approximation to the eigenvector associated with the smallest
eigenvalue of the infinite-dimensional matrix $\bH$ defined by \eqref{eq:H}.
We start by defining translational invariance for a (finite) tensor ring.

\subsection{Finite tensor ring}
\label{sec:TR}

The tensor ring decomposition~\cite{khor2011,zhzh2016,mika2020}
is a way to represent high-order tensors (or vectors that can be reshaped into high-order tensors)
by a sequence of 3rd-order tensors multiplied in circular fashion.

\begin{definition}[Tensor ring]
\label{def:TR}
Let $x_{\ell}$ be an $\ell$th-order tensor of size $d_1\times d_2\times\cdots\times d_{\ell}$. We say that $x_{\ell}$ is in \emph{tensor ring representation} if there exist $3$rd-order tensors $X_k$, $k=1,2\ldots,\ell$ of sizes $r_k\times d_k\times r_{k+1}$, with $r_{k+1}=r_1$, such that $x_{\ell}$ can be represented element-wise as $x_{\ell}(i_1,i_2,\ldots, i_{\ell}) = \trace\left(X_1(i_1)X_2(i_2)\cdots X_{\ell}(i_{\ell})\right)$ for $i_k\in\{1,2,\ldots,d_k\}$, where $X_k(i_k)$ denotes the $i_k$th lateral slice of $X_k$. The integers $r_1, r_2, \ldots, r_{\ell}$ are called \emph{tensor ring ranks}. Graphically, a tensor ring representation corresponds to
\[
x_\ell = \figname{TR}%
\mytikzTRing{$X_1$}{$X_2$}{$X_\ell$}{$i_1$}{$i_2$}{$i_\ell$}.
\]
\end{definition}

The 3rd-order tensors $X_k$ are called the \emph{cores} of $x_\ell$, with
$X_k(i_k)$ being the \emph{slice} of the $k$th core.
In the physics literature, a tensor ring is called a matrix product state
with periodic boundary conditions \cite{peve2007}.
When $r_1 = r_{\ell+1} = 1$, a tensor ring becomes a \emph{tensor train}~\eqref{eq:tt},
also known in the physics literature as a matrix product state
with open boundary conditions \cite{peve2007}.

When $d_k = d$, $r_k = r$, and $X_k(i_k) \equiv X(i_k)$ for every $k = 1,2,\ldots,\ell$,
the tensor ring $x_\ell$ is said to be translational invariant.
In this case, the cyclic property of the trace implies
\[
x_\ell(i_1,i_2,\ldots,i_\ell) = x_\ell(i_2,\ldots,i_\ell,i_1),
\]
that is, $x_\ell$ is invariant under cyclic permutations of its indices.
This is a desirable property because the eigenvectors of $\bH$ defined
in~\eqref{eq:H} are known to have such a property for certain $M_{k,k+1}$.
We define a translational invariant (finite) tensor ring as follows.

\begin{definition}[Translational invariant TR]
\label{def:tiTR}
Let $X$ be an $r\times d \times r$ tensor. We define a translational invariant (finite) tensor ring element-wise as $x_\ell(i_1,i_2,\ldots,i_\ell):= \trace \left[ X(i_1) X(i_2) \cdots X(i_\ell) \right]$ for $i_k\in\{1,2,\ldots,d\}$. We refer to $r$ as the (tensor ring) rank of $x_\ell$. Graphically, this representation corresponds to
\[
x_\ell = \figname{tiTR} \mytikzTRing{$X$}{$X$}{$X$}{$i_1$}{$i_2$}{$i_\ell$}.
\]
\end{definition}

Note that, due to the translational invariance, storing the tensor ring $x_\ell$ only requires to 
store a single core $X$ of size $r \times d \times r$.

\subsection{Infinite tensor ring}
\label{sec:iTR}

Practically, it is difficult to generalize \defref{def:TR} to tensor rings with an infinite number
of cores because this would require an infinite amount of data
to represent $\bfx$.
In contrast, it is possible to
generalize \defref{def:tiTR} to infinite tensor rings (iTR).

\begin{definition}[Translational invariant iTR]
\label{def:iTR}
Let $X$ be an $r\times d \times r$ tensor. We define a translational invariant infinite tensor ring element-wise as
\[
x(\ldots,i_{-1},i_0,i_1,\ldots) = \lim_{\ell\to \infty}\trace \left(\prod_{k=-\ell}^{\ell} X(i_k)\right),
\]
for $i_k\in\{1,2,\ldots,d\}$, assuming that the limit exists, and refer to $r$ as the rank of $x$. Graphically, we represent it as 
\[
\bfx = \figname{iTR1}\mytikziTR{$X$}{$i_{-1}$}{$i_0$}{$i_1$}.
\]
\end{definition}

Note that even though $\bfx$ contains an infinite number of cores, each core is completely defined by the $d$ \emph{slices} of $X$
so that $\bfx$ can be represented by a finite amount of data.
A translational invariant iTR is also known as a uniform matrix product state~\cite{zava2018}.
In the remainder of the paper, we will only work with translational invariant iTRs and just call them iTRs for brevity.

\subsection{Transfer matrix}

The core tensor $X$ of an iTR defines a special matrix called the \emph{transfer matrix},
which plays a key role in many operations performed on the iTR, including
Rayleigh quotient and residual calculations.

\begin{definition}[Transfer matrix]
\label{def:tf-matrix}
Let $\bfx$ be an iTR as in \defref{def:iTR}.
Then we define the transfer matrix $T_X$ associated with $\bfx$ as
the $r^2 \times r^2$ matrix
\begin{equation}
T_X := \sum_{i=1}^d X(i) \otimes X(i) = \,\figname{iTR1-TX}\mytikzT{$X$}\ ,
\label{eq:tf-matrix}
\end{equation}
where $X(i)$ is the $i$th slice of $X$.
\end{definition}

Multiplying $T_X$ with a vector $v \inR[r^2]$ is equivalent to performing the summation $W = \sum_{i = 1}^d X(i) V X(i)^T$, where $V$ is obtained from reshaping $v$ into an $r\times r$ matrix.
It is known that the transfer matrix \eqref{eq:tf-matrix} defines a complete positive map~\cite{choi1975,sanz2011}, which is a linear transformation that maps a positive semi-definite matrix $V$ to another positive semi-definite matrix $XVX^T$.
Under an additional irreducibility assumption, the dominant eigenvalue of $T_X$ is simple (and positive); if we reshape the corresponding left or right eigenvector into a matrix, that matrix is symmetric positive definite~\cite{evho1978,fana1992,peve2007}. In the following we will tacitly assume that the dominant eigenvalue is simple, which allows us to easily obtain the infinite power of $T_X$ from its left and right eigenvectors.

\begin{lemma}
\label{lem:tf-matrix-power}
Let $\eta>0$ be the simple dominant eigenvalue of $T_X$,
with $v_L$ and $v_R$ being the corresponding left and right eigenvectors,
respectively, normalized to satisfy $v_L\T v_R = 1$. Then
\begin{equation}
\lim_{k \to \infty} \left( \frac{T_X}{\eta} \right)^k = v_R v_L\T.
\label{eq:tf-matrix-lem}
\end{equation}
\end{lemma}
\begin{proof}
By the assumptions, there are nonsingular matrices of the form
$W_L = \begin{bmatrix} v_L & \bullet \end{bmatrix}$, $W_R = \begin{bmatrix} v_R & \bullet \end{bmatrix}$ such that
\begin{align}
W_L\H T_X &= \Lambda W_L\H, &
  T_X W_R &= W_R \Lambda, &
  \Lambda &= \begin{bmatrix} \eta & 0 \\ 0 & \Lambda_2 \end{bmatrix}, &
W_L\H W_R &= \eye,
\end{align}
where the spectral radius of $\Lambda_2$ is smaller than $1$.
Using that $W_R\inv = W_L\H$, this yields
\[
\lim_{k \to \infty} \left( \frac{T_X}{\eta} \right)^k
 = \lim_{k \to \infty} \frac{1}{\eta^k} W_R \Lambda^k W_R\inv
 = \lim_{k \to \infty} \frac{1}{\eta^k} W_R \Lambda^k W_L\H
 = v_R v_L\T,
\]
which completes the proof.
\end{proof}

\begin{corollary}[Normalized iTR]
\label{cor:normalized-iTR}
Let $\bfx$ be an iTR as in \defref{def:iTR}.
Then $\bfx$ is normalized, i.e., $\bfx\T \bfx = 1$,
if its corresponding transfer matrix $T_X$
has a simple dominant eigenvalue $\eta = 1$.
\end{corollary}
\begin{proof}
The proof directly follows from \thref{lem:tf-matrix-power}, i.e.,
\[
\bfx\T \bfx
 = \,\figname{iTR1-norm}\mytikziTRnorm{$X$}
 = \trace\left[ \lim_{k \to \infty} \left(T_X\right)^k \right]
 = \trace\left[ v_R v_L\T \right]
 = v_L\T v_R
 = 1,
\]
where we also used the trace property $\trace(AB) = \trace(BA)$.
\end{proof}

When the dominant eigenvalue of the transfer matrix $\eta$ is simple, $v_L\T v_R \neq 0$ holds.
We will assume that the eigenvectors $v_L$ and $v_R$ satisfy the normalization condition
$v_L\T v_R = 1$.
A graphical depiction of the dominant left and right eigenvalue relation, and
the normalization condition for $v_L =: \vecz(V_L)$,
                                $v_R =: \vecz(V_R)$ are shown in
\cref{fig:tfmtx-eig}.

\begin{figure}[hbtp]
\centering
\fignames{iTR1-TX-eig-}%
~\hfill%
$\myvcenter{\mytikzTleft{$X$}{$V_L$}}
 = \eta\ \myvcenter{\mytikzlefteigv{$V_L$}}$
\hfill%
$\myvcenter{\mytikzvLTvR{$V_L$}{$V_R$}} = 1$
\hfill%
$\myvcenter{\mytikzTright{$X$}{$V_R$}}
 = \eta\,\myvcenter{\mytikzrighteigv{$V_R$}}$
\hfill~%
\caption{Dominant eigenvalue relations of the transfer matrix
and the normalization condition $v_L\T v_R = 1$.
}
\label{fig:tfmtx-eig}
\end{figure}

\subsection{Canonical decomposition}
\label{sec:can}

Because we can insert the product of any nonsingular matrix $S$ and its
inverse between two consecutive cores of an iTR and redefine each slice
as $S^{-1}X(i)S$, $i=1,...,d$, the representation of an iTR is not unique.
It is useful to define a canonical form that makes it easier to describe
computations performed on iTRs\footnote{We use a slightly different definition of the canonical form compared to~\cite{zava2018} in order to ensure that a canonical iTR is always normalized to 1. The (mixed) canonical form introduced in \cite{zava2018} normalizes to $\normfrob{\Sigma^2}^2$.}.

\begin{definition}[Canonical form]
\label{def:iTR-can}
Let $\bfx$ be an iTR as in \defref{def:iTR}. We say that $\bfx$ is in canonical form if it can be represented as
\begin{equation}
\bfx(\ldots,i_{-1},i_0,i_1,\ldots)
  := \trace \left[ \prod_{k=-\infty}^{+\infty} Q(i_k) \Sigma \right],
\label{eq:iTR-can}
\end{equation}
and graphically
\[
\bfx = \,\figname{iTR1c}\mytikziTRc{$Q$}{$\Sigma$}{$i_{-1}$}{$i_0$}{$i_1$}
\]
where $Q(i) \inRR{r}$, for $i = 1,\ldots,d$, and $\Sigma \inRR{r}$ is a diagonal matrix with decreasing non-negative real numbers on its diagonal and $\normfrob{\Sigma} = 1$, such that the following left and right orthogonality conditions hold
\begin{align}
\sum_{i=1}^d Q_L(i)\T Q_L(i) &= \eta\eye, &
\sum_{i=1}^d Q_R(i) Q_R(i)\T &= \eta\eye,
\label{eq:can-condL-condR}
\end{align}
with $Q_L(i) := \Sigma Q(i)$, $Q_R(i) := Q(i) \Sigma$,
and $\eta \inR$ being the dominant eigenvalue of the transfer matrix $T_X$.
\end{definition}

We call the 3rd-order tensor $Q$ the orthogonal core of $\bfx$ and $\Sigma$ the corresponding singular value matrix, since it can be obtained from a singular value decomposition.
The 3rd-order tensors $Q_L$ and $Q_R$, defined in \eqref{def:iTR-can},
are called the left and right orthogonal cores, respectively.

Before explaining how to compute the canonical form, we will list some properties.
First, note that, by definition, $Q$ and $\Sigma$ satisfy the relations
\begin{equation}
\Sigma Q(i) \Sigma = Q_L(i) \Sigma = \Sigma Q_R(i),
\label{eq:rel-QLR}
\end{equation}
for $i = 1,\ldots,d$.
The left and right canonical transfer matrices are defined as follows.

\begin{definition}[Canonical transfer matrices]
\label{def:can-tf-matrix}
Let $\bfx$ be an iTR in canonical form as in \defref{def:iTR-can}.
Then we define its left and right canonical transfer matrices $T_{Q_L}$ and $T_{Q_R}$ as the following $r^2 \times r^2$ matrices
\begin{align}
\TQL &:= \sum_{i=1}^d Q_L(i) \otimes Q_L(i)\ = \,
\figname{iTR1c-TQL}\mytikzTL{$Q$}{$\Sigma$}\ ,
\label{eq:can-tf-matrix-left} \\[5pt]
\TQR &:= \sum_{i=1}^d Q_R(i) \otimes Q_R(i)\ = \,
\figname{iTR1c-TQR}\mytikzTR{$Q$}{$\Sigma$}\ ,
\label{eq:can-tf-matrix-right}
\end{align}
with $Q_L$ and $Q_R$ defined as in \eqref{def:iTR-can}.
\end{definition}

The canonical transfer matrices, associated with an iTR in canonical form,
as well as their dominant left and right eigenvectors,
have the following fixed form.

\begin{lemma}
\label{lem:can-tf-matrix-eigv}
Let $\bfx$ be an iTR in canonical form as in \defref{def:iTR-can} with its corresponding canonical transfer matrices as in \defref{def:can-tf-matrix}.
Then,
\begin{align}
\vecz(\eye)\T \TQL &= \eta \vecz(\eye)\T, &
\TQR\vecz(\eye) &= \eta \vecz(\eye), \label{eq:TQ-vecI} \\
\TQL\vecz(\Sigma^2) &= \eta \vecz(\Sigma^2), &
\vecz(\Sigma^2)\T \TQR &= \eta \vecz(\Sigma^2)\T, \label{eq:TQ-vecS2}
\end{align}
with $\vecz(\eye)$ being the dominant left eigenvector of $\TQL$
and the dominant right eigenvector of $\TQR$,
and $\vecz(\Sigma^2)$ being the dominant right eigenvector of $\TQL$
and the dominant left eigenvector of $\TQR$.
\end{lemma}
\begin{proof}
The proof for the dominant right eigenvector of $\TQR$ directly follows from
vectorizing the second equality of \eqref{eq:can-condL-condR}:
\begin{equation*}
\eta \vecz(\eye) = \vecz\left( \sum_{i=1}^d Q_R(i) Q_R(i)\T \right)
 = \left( \sum_{i=1}^d Q_R(i) \otimes Q_R(i) \right) \vecz(\eye)
 = \TQR \vecz(\eye),
\end{equation*}
where we used~\eqref{eq:can-tf-matrix-right} as well as relations between Kronecker products and vectorization:
\begin{equation}
\vecz(ABC) = (C\T \otimes A) \vecz(B).
\label{eq:vectrick}
\end{equation}
For the proof of the dominant left eigenvector of $\TQL$, we start from
vectorizing the first equality of \eqref{eq:can-condL-condR} and again apply
\eqref{eq:vectrick} to obtain
\begin{equation}
\eta \vecz(\eye) = \vecz\left( \sum_{i=1}^d Q_L(i)\T Q_L(i) \right)
 = \left( \sum_{i=1}^d Q_L(i)\T \otimes Q_L(i)\T \right) \vecz(\eye).
\label{eq:can-proof-TQL}
\end{equation}
Transposing \eqref{eq:can-proof-TQL} and using \eqref{eq:can-tf-matrix-left}
completes the proof of \eqref{eq:TQ-vecI}.

In order to prove that $\vecz(\Sigma^2)$ is the dominant right eigenvector of
$\TQL$, we start again from the vectorization of the right equality of
\eqref{eq:can-condL-condR},
\begin{equation}
\vecz\left( \sum_{i=1}^d Q(i) \Sigma^2 Q(i)\T \right)
 = \left( \sum_{i=1}^d Q(i) \otimes Q(i) \right) \vecz(\Sigma^2)
 = \eta \vecz(\eye),
\label{eq:can-proof-TQR}
\end{equation}
where we substituted $Q_R(i) = Q(i)\Sigma$ and applied \eqref{eq:vectrick}.
Next, we multiply both sides of \eqref{eq:can-proof-TQR} from the left by $\Sigma \otimes \Sigma$, yielding
\begin{equation*}
\underbrace{\left( \Sigma \otimes \Sigma \right)
\left( \sum_{i=1}^d Q(i) \otimes Q(i) \right)}_{=\ \TQL} \vecz(\Sigma^2)
 = \eta \left( \Sigma \otimes \Sigma \right) \vecz(\eye)
 = \eta \vecz(\Sigma^2).
\end{equation*}
Finally, in order to prove that $\vecz(\Sigma^2)$ is also the dominant left
eigenvector of $\TQR$, we start from the vectorization of the left equality
of \eqref{eq:can-condL-condR},
\begin{equation}
\vecz\left( \sum_{i=1}^d Q(i)\T \Sigma^2 Q(i) \right)
 = \left( \sum_{i=1}^d Q(i)\T \otimes Q(i)\T \right) \vecz(\Sigma^2)
 = \eta \vecz(\eye),
\label{eq:can-proof-TQL2}
\end{equation}
where we substituted $Q_L(i) = \Sigma Q(i)$ and made use of the identity \eqref{eq:vectrick}.
Next, we multiply both sides of \eqref{eq:can-proof-TQL2} from the left by $\Sigma \otimes \Sigma$ followed by a transposition, yielding
\begin{equation*}
\vecz(\Sigma^2)\T \underbrace{\left( \sum_{i=1}^d Q(i) \otimes Q(i) \right)
\left( \Sigma \otimes \Sigma \right)}_{=\ \TQR}
 = \eta \left[\left( \Sigma \otimes \Sigma \right) \vecz(\eye) \right]\T
 = \eta \vecz(\Sigma^2)\T,
\end{equation*}
which proves \eqref{eq:TQ-vecS2} and completes the proof.
\end{proof}

\begin{figure}[b!]
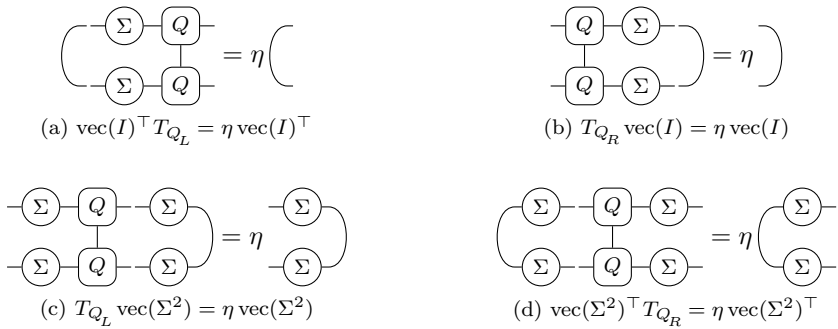

\centering
\fignames{iTR1c-TQ-eig-}%
\hfill%
%
\subfloat[$\vecz(I)\T \TQL = \eta \vecz(I)\T$\label{fig:can-TQL-left}]%
{$\quad\myvcenter{\mytikzTQLleft{$Q$}{$\Sigma$}}
 = \eta\,\myvcenter{\mytikzvecIT}\quad$}%
\hfill%
\hfill%
%
\subfloat[$\TQR \vecz(I) = \eta \vecz(I)$\label{fig:can-TQR-right}]%
{$\quad\myvcenter{\mytikzTQRleft{$Q$}{$\Sigma$}}
 = \eta\,\myvcenter{\mytikzvecI}\quad$}%
\hfill~\\[5pt]%
\hfill%
%
\subfloat[$\TQL\vecz(\Sigma^2) = \eta \vecz(\Sigma^2)$\label{fig:can-TQL-right}]%
{$\myvcenter{\mytikzTQLright{$Q$}{$\Sigma$}}
 = \eta\,\myvcenter{\mytikzvecSS{$\Sigma$}}$}%
\hfill%
\hfill%
%
\subfloat[$\vecz(\Sigma^2)\T \TQR = \eta \vecz(\Sigma^2)\T$\label{fig:can-TQR-left}]%
{$\myvcenter{\mytikzTQRright{$Q$}{$\Sigma$}}
 = \eta\,\myvcenter{\mytikzvecSST{$\Sigma$}}$}%
\hfill%
\caption{Eigenvalue decompositions of the canonical transfer matrices.}
\label{fig:can-tf-matrix-eigv}
\end{figure}

The eigenvalue decompositions of \lemref{lem:can-tf-matrix-eigv} are
graphically shown in \figref{fig:can-tf-matrix-eigv}.
Note also that the left and right orthogonality conditions
\eqref{eq:can-condL-condR}, defined by the canonical form,
can be expressed in terms of the canonical transfer matrices and
are equivalent to \eqref{eq:TQ-vecI}.
Using the orthogonality conditions \eqref{eq:can-condL-condR}, we notice that
computing the canonical form of an iTR corresponds to finding matrices $Q(i)$
and a diagonal $\Sigma$ so that
\begin{equation}
\sum_{i=1}^d Q(i)\T \Sigma^2 Q(i) = \eta\eye = \sum_{i=1}^d Q(i) \Sigma^2 Q(i)\T.
\label{eq:gramians}
\end{equation}
\algref{alg:can} shows how the canonical decomposition of an iTR can be
computed \cite{orvi2008}.
The complexity of \algref{alg:can} is dominated by step~1, which requires the
computation of the dominant left and right eigenvectors of the $r^2 \times r^2$
transfer matrix $T_X$.
In step~2, we use that the matricization of the dominant left and right eigenvectors, $V_L$ and $V_R$, are positive definite matrices \cite{fana1992}.
Using the properties $L\inv = R \Sigmat$, $R\inv = \Sigmat L$,
$V_R = R \Sigmat^2 R\T$, and $V_L = L\T \Sigmat^2 L$,
in combination with the identity \eqref{eq:vectrick},
we can show that $Q$ and $\Sigma$ obtained in step~6 of \algref{alg:can}
satisfy \eqref{eq:gramians}.
Finally, we normalize the canonical iTR in step~7.

\begin{algorithm2e}[hbtp]
\caption{Canonical decomposition of iTR}%
\label{alg:can}%
\SetKwInOut{input}{Input}
\SetKwInOut{output}{Output}
\BlankLine
\input{3rd-order tensor $X$}
\output{3rd-order tensor $Q$ and diagonal matrix $\Sigma$}
\BlankLine
\nl Dominant left and right eigenvectors of transfer matrix $T_X$:
\[
\mathrm{vec}(V_L)\T\,T_X = \eta\,\mathrm{vec}(V_L)\T, \qquad
  T_X\,\mathrm{vec}(V_R) = \eta\,\mathrm{vec}(V_R).
\]\\
\nl
\begin{minipage}{0.45\textwidth}
Eigendecompositions of $V_L$ and $V_R$:
\end{minipage}\hfill%
\begin{minipage}{0.45\textwidth}\centering
\(
V_L = U_L \Lambda_L U_L\T, \quad V_R = U_R \Lambda_R U_R\T.
\)
\end{minipage}\\[5pt]
\nl
\begin{minipage}{0.45\textwidth}
Form:
\end{minipage}\hfill%
\begin{minipage}{0.45\textwidth}\centering
\(
\Ut_L = U_L \Lambda_L^{1/2}, \quad \Ut_R = U_R \Lambda_R^{1/2}.
\)
\end{minipage}\\[5pt]
\nl
\begin{minipage}{0.45\textwidth}
Singular value decomposition:
\end{minipage}\hfill%
\begin{minipage}{0.45\textwidth}\centering
\(
V \Sigmat W\T = \Ut_L\T \Ut_R.
\)
\end{minipage}\\[5pt]
\nl
\begin{minipage}{0.45\textwidth}
Form:
\end{minipage}\hfill%
\begin{minipage}{0.45\textwidth}\centering
\(
L = W\T \Ut_R\inv, \quad R = \Ut_L\invT V.
\)
\end{minipage}\\[5pt]
\nl
\begin{minipage}{0.45\textwidth}
Canonical form:
\end{minipage}\hfill%
\begin{minipage}{0.45\textwidth}\centering
\(
Q(i) = \normfrob{\Sigmat} \, L X(i) R, \quad
\Sigma = \Sigmat/\normfrob{\Sigmat}.
\)
\end{minipage}\\[5pt]
\nl
\begin{minipage}{0.45\textwidth}
Normalization:
\end{minipage}\hfill%
\begin{minipage}{0.45\textwidth}\centering
\(
Q(i) = Q(i)/\sqrt{\eta}.
\)
\end{minipage}\\[5pt]
\end{algorithm2e}

\subsection{Frame matrix}

Starting from an iTR in canonical form as in \defref{def:iTR-can},
we can rewrite it as a product of an infinite-dimensional matrix with
orthogonal columns, called the \emph{frame matrix} \cite{dokh2014,krst2014},
and a finite vector of length equal to the number of elements in 1 core.

\begin{definition}[Frame matrix]
\label{def:iTR-frame-matrix}
Let $\bfx$ be a normalized iTR in canonical form as in \defref{def:iTR-can}.
Then we define its $k$th frame matrix as follows
\begin{equation}
\bF_{\neq k} := \,\figname{iTR1c-frame}%
\mytikziTRframe{$Q$}{$\Sigma$}{i}{j}{k-2}{k-1}{k}{k+1}\ .
\label{eq:iTR-frame}
\end{equation}
\end{definition}
Note that by this definition the frame matrix $\bF_{\neq k}$ has an infinite
number of rows, which are represented as downwards pointing legs, and $dr^2$ columns,
which are represented as upwards pointing legs.
Now the iTR $\bfx$ can also be written as
\fignames{iTR1c-center-core-}%
\begin{align}
                 \bfx &= \bF_{\neq k} \vecz(Q_\can), &
\mytikzcore{$Q_\can$} &= \mytikzcentercore{$\Sigma$}{$Q$}{$\Sigma$}\,,
\label{eq:bfx-frame-center-core}
\end{align}
where $Q_\can$ is called the \emph{canonical center core}.
Another important property of the frame matrix is that it has orthonormal
columns, i.e., $\bF_{\neq k}\T \bF_{\neq k} = \eye[dr^2]$, since
\begingroup
\allowdisplaybreaks
\begin{align*}
\bF_{\neq k}\T \bF_{\neq k}
 &= \,\figname{iTR1c-frame-orth1}\mytikziTRframeorthA{$Q$}{$\Sigma$}\ ,\\[5pt]
 &= \,\figname{iTR1c-frame-orth2}\mytikziTRframeorthB{$Q$}{$\Sigma$}\ ,\\[5pt]
 &= \,\figname{iTR1c-frame-orth3}\mytikziTRframeorthC{$Q$}{$\Sigma$}\,
  = \eye[dr^2],
\end{align*}
\endgroup
where we used \lemref{lem:tf-matrix-power} in the first step, i.e.,
$\lim_{k \to \infty} (\TQR)^k = \vecz(\eye) \vecz(\Sigma^2)\T$,
in order to replace the infinite repetition of $\TQR$ on both sides of the
free indices, followed by
$\vecz(\eye)\T \TQL = \vecz(\eye)\T$, and
$\vecz(\Sigma^2)\T \vecz(\eye) = 1$, respectively.

Due to the orthogonality of the frame matrix, this matrix will turn out
to be an important tool to extract and update cores in an efficient way.
Further, this matrix will also play a role in splitting up supercores.

\subsection{2-core translational invariance}
\label{sec:iTR2}

The translational invariance principle can be
generalized to allow the unit of invariance to include more than one core. Such
a unit is sometimes referred to as a \emph{supercore}.
We now define a 2-core translational invariant iTR\cite{zava2018}, which we will refer to as iTR2
to distinguish it from iTRs with a single core.

\begin{definition}[2-core translational invariant iTR]
\label{def:iTR2}
An infinite-dimensional vector $\bfx$ is said to be in
\emph{2-core translational invariant iTR format} if it can be represented as
\begin{equation}
\bfx(\ldots,i_0,i_1,i_2,i_3,\ldots)
  := \trace \left[ \prod_{k=-\infty}^{+\infty} X(i_{2k})\,Y(i_{2k+1}) \right],
\label{eq:iTR2}
\end{equation}
and graphically
\[
\bfx = \,\figname{iTR2}\mytikziTRR{$X$}{$Y$}{$i_0$}{$i_1$}{$i_2$}{$i_3$}
\]
where all indices $i_k$ run from 1 to $d$ and each $X(i)$ and $Y(i)$ are
matrices of size $r \times r$.
\end{definition}

In contrast to an iTR with a single core, there are now 2 transfer matrices associated with
every iTR2, playing a central role in the iTR2 operations.

\begin{definition}[iTR2 transfer matrices]
\label{def:tf-matrices}
Let $\bfx$ be an iTR2 as in \defref{def:iTR2}.
Then we define the transfer matrices $T_{XY}$ and $T_{YX}$ associated with
$\bfx$ as the $r^2 \times r^2$ matrices
\begin{align}
T_{XY} &:= \sum_{i=1}^d \sum_{j=1}^d \Big[ X(i) \otimes X(i) \Big]
                                     \Big[ Y(j) \otimes Y(j) \Big] = \,
\figname{iTR2-TXY}\mytikzTT{$X$}{$Y$}\ , \label{eq:tf-matrix-XY}\\[5pt]
T_{YX} &:= \sum_{i=1}^d \sum_{j=1}^d \Big[ Y(i) \otimes Y(i) \Big]
                                     \Big[ X(j) \otimes X(j) \Big] = \,
\figname{iTR2-TYX}\mytikzTT{$Y$}{$X$}\ , \label{eq:tf-matrix-YX}
\end{align}
where $X(i)$ is the $i$th slice of $X$ and $Y(j)$ the $j$th slice of $Y$.
\end{definition}

Note that the transfer matrices $T_{XY}$ and $T_{YX}$ of an iTR2
can also be expressed in terms of the 1-core transfer matrices, i.e.,
$T_{XY} = T_X T_Y$ and $T_{YX} = T_Y T_X$,
where $T_X$ is the transfer matrix formed by only using the core $X$
and $T_Y$ the transfer matrix formed by only using the core $Y$.

\begin{lemma}
\label{lem:can2-tf-matrices}
Let $\bfx$ be an iTR2 as in \defref{def:iTR2}
and assume that $T_X$ and $T_Y$ are nonsingular.
Then its corresponding transfer matrices $T_{XY}$ and $T_{YX}$,
defined in \defref{def:tf-matrices}, have the same eigenvalues.
\end{lemma}
\begin{proof}
We will show that every eigenvalue of $T_{XY}$ is also an eigenvalue $T_{YX}$,
and vice versa.
First, assume that $(\lambda,v)$ is an eigenpair of $T_{XY}$,
\(
T_{XY} v = T_X T_Y v = \lambda v.
\)
Multiplying both sides from the left by $T_Y$, yields
\(
T_Y T_X (T_Y v) = \lambda (T_Y v),
\)
hence, the pair $(\lambda,w := T_Y v)$ is an eigenpair of $T_{YX}$.

Similarly, assume that $(\lambda,w)$ is an eigenpair of $T_{YX}$,
\(
T_{YX} w = T_Y T_X w = \lambda w.
\)
Multiplying both sides from the left by $T_X$, yields again
\(
T_X T_Y (T_X w) = \lambda (T_X w),
\)
hence, the pair $(\lambda,v := T_X w)$ is an eigenpair of $T_{XY}$.
\end{proof}

We now define the canonical form of an iTR2.
Since the canonical form is mainly used for normalized iTRs and iTR2s,
we will assume that the dominant eigenvalue of its transfer matrices $\eta = 1$.

\begin{definition}[iTR2 canonical form]
\label{def:iTR-can2}
Let $\bfx$ be an iTR2 as in \defref{def:iTR2}, then $\bfx$ is in canonical form if it can be represented as
\begin{equation}
\bfx(\ldots,i_0,i_1,i_2,i_3,\ldots)
  := \trace \left[ \prod_{k=-\infty}^{+\infty} Q(i_{2k})\,\Sigma\,U(i_{2k+1})\,\Omega\,\right],
\label{eq:iTR-can2}
\end{equation}
and graphically
\[
\bfx = \,\figname{iTR2c}%
\mytikziTRRc{$Q$}{$\Sigma$}{$U$}{$\Omega$}{$i_0$}{$i_1$}{$i_2$}{$i_3$}
\]
where $Q(i),U(i) \inRR{r}$, for $i = 1,\ldots,d$, and $\Sigma,\Omega \inRR{r}$ are diagonal matrices with decreasing non-negative real numbers on its diagonal and $\normfrob{\Sigma} = \normfrob{\Omega} = 1$, such that the following left and right orthogonality conditions hold
\begin{align}
\sum_{i=1}^d Q_L(i)\T Q_L(i) &= \eye, &
\sum_{i=1}^d Q_R(i) Q_R(i)\T &= \eye, \label{eq:can2-condLR1} \\
\sum_{i=1}^d U_L(i)\T U_L(i) &= \eye, &
\sum_{i=1}^d U_R(i) U_R(i)\T &= \eye, \label{eq:can2-condLR2}
\end{align}
with $Q_L(i) := \Omega Q(i)$, $Q_R(i) := Q(i) \Sigma$, $U_L(i) := \Sigma U(i)$,
and $U_R(i) := U(i) \Omega$.
\end{definition}

Due to a translational invariance of 2 cores for an iTR2, its canonical form
gives rise to 4 canonical transfer matrices with structured dominant
eigenvectors as detailed in the following definition and lemma.

\begin{definition}[iTR2 canonical transfer matrices]
\label{def:can2-tf-matrix}
Let $\bfx$ be an iTR2 in canonical form as in \defref{def:iTR-can2}.
Then we define its left canonical transfer matrices as the $d^2 \times d^2$ matrices $\TQUL$ and $\TUQL$,
and its right canonical transfer matrices as $\TQUR$ and $\TUQR$, i.e.,
\begin{align}
\TQUL &= \TQL \TUL, &
\TUQL &= \TUL \TQL, &
\TQUR &= \TQR \TUR, &
\TUQR &= \TUR \TQR,
\end{align}
where $Q_L$, $Q_R$, $U_L$, and $U_R$ are given in \defref{def:iTR-can2}.
\end{definition}

\begin{lemma}
\label{lem:can2-tf-matrix-eigv}
Let $\bfx$ be an iTR2 in canonical form as in \defref{def:iTR-can2} and
its corresponding canonical transfer matrices as in \defref{def:can2-tf-matrix}.
Then,
\begin{align}
  \TQUL \vecz(\Omega^2) &= \vecz(\Omega^2), &
\vecz(\Omega^2)\T \TQUR &= \vecz(\Omega^2)\T, \label{eq:TQU-eigv} \\
  \TUQL \vecz(\Sigma^2) &= \vecz(\Sigma^2), &
\vecz(\Sigma^2)\T \TUQR &= \vecz(\Sigma^2)\T, \label{eq:TUQ-eigv}
\end{align}
or $\vecz(\Omega^2)$ being the dominant right eigenvector of $\TQUL$
and the dominant left eigenvector $\TQUR$,
and $\vecz(\Sigma^2)$ being the dominant right eigenvector of $\TUQL$
and the dominant left eigenvector $\TUQR$.
\end{lemma}
\begin{proof}
The proof is similar to the one of \lemref{lem:can-tf-matrix-eigv}.
\end{proof}

The canonical decomposition of an iTR2 can be computed by making use of
\algref{alg:can} applied to the supercore, followed by a singular value
decomposition of the matricization of the obtained canonical center supercore.
This procedure is summarized in \algref{alg:can2}.
In case $d \ll r$, the complexity of \algref{alg:can2} is dominated by step 2,
which requires again computing the dominant left and right eigenvectors of the
$r^2 \times r^2$ transfer matrix in \algref{alg:can}.

\begin{algorithm2e}[hbtp]
\caption{Canonical decomposition of iTR2}%
\label{alg:can2}%
\SetKwInOut{input}{Input}
\SetKwInOut{output}{Output}
\BlankLine
\input{3rd-order tensors $X$ and $Y$}
\output{3rd-order tensors $Q$ and $U$,
        and diagonal matrices $\Sigma$ and $\Omega$}
\BlankLine
\fignames{alg-iTR2c-}%
\nl Form supercore $\cZ$ as the contraction of the cores $X$ and $Y$. \\
\nl Use \algref{alg:can} on $\cZ$ resulting in the supercore $\cQ$ and
    the diagonal matrix $\Omega$. \\
\nl Form canonical center supercore $\cQ_\can$ and reshape into a matrix $M$:
\[
\mytikzcentersupercorecombined{$\Omega$}{$\cQ$}{$\Omega$}\, = \,
\mytikzsupercoreind{$\cQ_\can$}{$j_1$}{$i_1i_2$}{$j_2$}
\qquad \longrightarrow \qquad
\mytikzsupermatind{$M$}{$j_1i_1$}{$j_2i_2$}\,.
\]\\
\nl Singular value decomposition of $M$ and reshape back into cores:
\[
V\,\Sigma\,W\T = M \qquad \longrightarrow \qquad
\mytikzTMTind{$V$}{$\Sigma$}{$W$}{$j_1$}{$i_1$}{$i_2$}{$j_2$}\,.
\]\\
\nl Canonical form: \quad
\(
Q(i) = \Omega\inv V(i), \quad U(j) = W(j) \Omega\inv.
\)
\end{algorithm2e}

\section{Eigenvalue approximation}
\label{sec:eigval-approx}

In this section, we define and study appropriate (averaged) notions of the Rayleigh quotient and the residual for approximating the (averaged) eigenvalues of the infinite-dimensional matrix $\bH$ defined in~\eqref{eq:H}.
For this purpose, it will be useful to consider the graphical notation of 
$\bH_k = \cdots I \otimes M_{k,k+1}
               \otimes I \cdots$,
the summand of $\bH$ that acts on sites $k$ and $k+1$:
\[
\bH_k = \,\cdots\quad \myvcenter{\figname{bHk}\mytikzbHk{$\qquad M\qquad$}{k-2}{k-1}{k}{k+1}{k+2}{k+3}}\ \cdots,
\]
where $M_{k,k+1} \equiv M \inRR{d^2}$.
Using this notation, it can be observed that the infinite-dimensional matrix-vector product
$\bH_k \bfx$, where $\bfx$ is an iTR, corresponds to contracting
along all $i$ indices.

\subsection{Rayleigh quotient per site}
\label{sec:rq}

Let $\bfx$ be a normalized iTR that constitutes an approximation to an eigenvector of $\bH$. The \emph{Rayleigh quotient per site} of
$\bH$ with respect to $\bfx$ can be defined as the following limit
\begin{equation} \label{eq:deftheta}
\theta := \lim_{\ell \to \infty} \frac{x_\ell\T \cH_\ell x_\ell}{2\ell}
        = \lim_{\ell \to \infty} \frac{1}{2\ell}
          \sum_{k=-\ell}^{\ell-1} x_\ell\T \Ht_k x_\ell.
\end{equation}
Here, $\bH_\ell$ is the finite-dimensional Hamiltonian defined in~\eqref{eq:Hell} and 
$x_\ell$ is the finite tensor ring (with the same core tensor $X$) from \cref{def:tiTR}. Because of translational invariance, each summand $x_\ell\T \Ht_k x_\ell$ is actually identical and, therefore, the limit exists. This limit is equal to 
\begin{equation}
\theta  = \bfx\T\bH_{k}\bfx = \,\myvcenter{\figname{iTR1-rq-inf}\mytikzrqinf{$X$}{$X$}}\,,
\label{eq:rq-inf}
\end{equation}
for arbitrary $k$. Once again because of translational invariance, the value of $\bfx\T\bH_{k}\bfx$ is independent of $k$. The quantity~\eqref{eq:rq-inf} can be viewed as a Rayleigh quotient \emph{per site} and we will simply refer to it as the iTR Rayleigh quotient in the following.
The following theorem describes how such a Rayleigh quotient can be calculated
in terms of the eigenvectors of the transfer matrix $T_X$ associated with $\bfx$.

\begin{theorem}[iTR Rayleigh quotient]
\label{th:iTR-rq}
Let $\bfx$ be a normalized iTR as defined in \defref{def:iTR},
with the corresponding transfer matrix $T_X$ from \defref{def:tf-matrix}.
Then the iTR Rayleigh quotient for an infinite-dimensional matrix $\bH$
of the form \eqref{eq:H} can be computed as follows
\begin{equation}
\theta = \,\myvcenter{\figname{iTR1-rq}\mytikzrq{$X$}{$X$}{$V_L$}{$V_R$}}\ ,
\label{eq:rq}
\end{equation}
where $V_L$ and $V_R$ are, respectively, the matricizations of the left and
right dominant eigenvectors of $T_X$. 
\end{theorem}
\begin{proof}
We first rewrite \eqref{eq:rq-inf} in matrix notation
\begin{align}
\bfx\T\bH_{k}\bfx
 &= \trace\left[ \left( \prod_{\ell=-\infty}^{k-1} T_X \right)
                 \Mt
                 \left( \prod_{\ell=k+2}^{+\infty} T_X \right) \right],&
\Mt &:=
 \,\myvcenter{\figname{iTR1-rq-Mt}\mytikzMt{$X$}{$X$}}\,.
\label{eq:iTR-rq-Mt}
\end{align}
It follows from \lemref{lem:tf-matrix-power} that
\begin{equation*}
\bfx\T\bH_{k}\bfx
 = \trace\left[ v_R v_L\T\,\Mt\,v_R v_L\T \right]
 = \trace\left[ v_L\T\,\Mt\,v_R v_L\T v_R \right]
 = v_L\T \Mt v_R
 = \theta,
\end{equation*}
where we consecutively used the trace cyclic property $\trace(XYZ) = \trace(YZX)$
and the normalization of the left and right eigenvectors $v_L\T v_R = 1$.
\end{proof}

If $\bfx$ is a normalized iTR, the evaluation of the Rayleigh quotient only
requires multiplying the matrix $\Mt$ \eqref{eq:iTR-rq-Mt} from the left
and right, respectively, with the left and right dominant eigenvectors of its
corresponding transfer matrix $T_X$.
In case $\bfx$ is given in canonical form, the Rayleigh quotient simplifies
further because there is no need to compute the left and right dominant
eigenvectors of the transfer matrix, see \lemref{lem:can-tf-matrix-eigv}.

\begin{corollary}[Canonical iTR Rayleigh quotient]
\label{cor:iTRc-rq}
Let $\bfx$ be a normalized iTR in canonical form as in
\defref{def:iTR-can} and its corresponding canonical transfer matrices as in
\defref{def:can-tf-matrix}.
Then the iTR Rayleigh quotient associated with the infinite-dimensional
matrix $\bH$ \eqref{eq:H} can be computed as follows
\begin{equation}
\theta = \,\myvcenter{\figname{iTR1c-rq}\mytikzrqcan{$Q$}{$Q$}{$\Sigma$}{$\Sigma$}}\ ,
\label{eq:iTRc-rq}
\end{equation}
where $Q$ and $\Sigma$ are given in \defref{def:iTR-can}.
\end{corollary}
\begin{proof}
Since $\bfx$ is in canonical form, the iTR Rayleigh quotient~\eqref{eq:rq-inf} takes the form
\begin{equation}
\bfx\T\bH_{k}\bfx
 = \,\myvcenter{\figname{iTR1c-rq-inf}\mytikzrqinfcan{$Q$}{$Q$}{$\Sigma$}{$\Sigma$}}\,,
\label{eq:rq-inf-can}
\end{equation}
for arbitrary $k$. In matrix notation, this reads as
\[
\bfx\T\bH_{k}\bfx
 = \trace\left[ \left( \prod_{\ell=-\infty}^{k-1} \TQR \right)
                \Mt_Q
                \left( \prod_{\ell=k+2}^{+\infty} \TQR \right) \right], \qquad
\Mt_Q := \,
  \myvcenter{\figname{iTR1c-rq-MtQ}\mytikzMtQ{$Q$}{$Q$}{$\Sigma$}{$\Sigma$}}\,.
\]
Using \lemref{lem:tf-matrix-power} together with the fact that $\vecz(\Sigma^2)$ and
$\vecz(\eye)$ are the left and right dominant eigenvectors of
$\TQR$, respectively, we obtain
\[
\bfx\T\bH_{k}\bfx
 = \trace\left[ \vecz(\eye) \vecz(\Sigma^2)\T\,\Mt_Q\,
                \vecz(\eye) \vecz(\Sigma^2)\T \right]
 = \vecz(\Sigma^2)\T\,\Mt_Q\,\vecz(\eye)
 = \theta,
\]
where we consecutively used the trace cyclic property $\trace(XYZ) = \trace(YZX)$
and the normalization of the left and right eigenvectors of $\TQR$, i.e.,
$\vecz(\Sigma^2)\T \vecz(\eye) = 1$.
\end{proof}

If $\bfx$ is an iTR2, similar arguments as above show that the limit~\eqref{eq:deftheta} equals to
\begin{equation} \label{eq:itr2def}
  \theta = \frac12 \bfx\T\bH_{2k}\bfx + \frac12 \bfx\T\bH_{2k+1}\bfx,
\end{equation}
for an arbitrary integer $k$.
The even and odd terms take the following form:
\begin{align}
\bfx\T\bH_{2k}\bfx &= \,\myvcenter{\figname{iTR2-rq-inf-even}%
\mytikzrqinf{$X$}{$Y$}}\,,
\label{eq:rq2-inf-even}\\[5pt]
\bfx\T\bH_{2k+1}\bfx &= \,\hspace{8mm}\myvcenter{\figname{iTR2-rq-inf-odd}%
\mytikzrqinf{$Y$}{$X$}}\,.
\label{eq:rq2-inf-odd}
\end{align}

\begin{theorem}[iTR2 Rayleigh quotient]
\label{th:iTR2-rq}
Let $\bfx$ be a normalized nonzero iTR2 as in \defref{def:iTR2}
and its corresponding transfer matrices be $T_{XY}$ and $T_{YX}$ as in
\defref{def:tf-matrices}.
Then the Rayleigh quotient~\eqref{eq:itr2def}  can be computed as follows
\begin{equation}
\theta = \frac{1}{2}
         \,\myvcenter{\figname{iTR2-rq-even}\mytikzrq{$X$}{$Y$}{$V_L^\even$}{$V_R^\even$}}\,
      +\,\frac{1}{2}
         \,\myvcenter{\figname{iTR2-rq-odd}\mytikzrq{$Y$}{$X$}{$V_L^\odd$}{$V_R^\odd$}}\ ,
\label{eq:rq2}
\end{equation}
where $V_L^\even$ and $V_R^\even$ are the matricizations of the left and right
dominant eigenvectors of $T_{XY}$, and $V_L^\odd$ and $V_R^\odd$ are the
matricization of the left and right dominant eigenvectors of $T_{YX}$,
respectively.
\end{theorem}
\begin{proof}
The proof proceeds analogously to the one for \thref{th:iTR-rq}.
\end{proof}
Using the iTR2 canonical form of \defref{def:iTR-can2}, the Rayleigh quotient
\eqref{eq:rq2} amounts to
\begin{equation}
\theta = \frac{1}{2}\,\myvcenter{\figname{iTR2c-rq-even}%
\mytikzrqcan{$Q$}{$U$}{$\Sigma$}{$\Omega$}}\,
      +\,\frac{1}{2}\,\myvcenter{\figname{iTR2c-rq-odd}%
\mytikzrqcan{$U$}{$Q$}{$\Omega$}{$\Sigma$}}\ ,
\label{eq:iTR2c-rq}
\end{equation}
where $Q$, $U$, $\Sigma$, and $\Omega$ are given in \defref{def:iTR-can2}.

\subsection{Residual}
\label{sec:res}

One way to measure the accuracy of an approximate eigenpair $(\theta, \bfx)$
of $\bH$ is to evaluate the residual norm.
When $\bH$ is infinite-dimensional and $\bfx$ an iTR, some care is needed to
properly define this residual.
By making use of the property that $\bfx$ can be written as the matrix-vector
product of the frame matrix and the canonical center core
\eqref{eq:bfx-frame-center-core},
we introduce a local iTR residual involving only one core of $\bfx$.

We start by projecting $\bH$ into the subspace spanned by the columns of
the frame matrix, defined in \defref{def:iTR-frame-matrix}.
Due to the translational invariance of $\bfx$, we can choose an arbitrary $k$
for the frame matrix.
Therefore, without loss of generality,
we set $k = 0$ for the remainder of this section.
We define the projection of $\bH$ onto $\bF_{\neq0}$ as follows
\begin{equation}
H_{\neq0} := \bF_{\neq0}\T \bH \bF_{\neq0}
           = \sum_{k=-\infty}^{+\infty} \bF_{\neq0}\T \bH_k \bF_{\neq0}
           =: \sum_{k=-\infty}^{+\infty} H_k.
\label{eq:Htsum}
\end{equation}
where $\bH_k$ is defined in \eqref{eq:H} and the matrices $H_k \inRR{dr^2}$
are obtained by projecting the corresponding $\bH_k$ onto $\bF_{\neq0}$.
Depending on the relative position of the matrix $M$ in each $\bH_k$, with
respect to index $i_0$ in $\bF_{\neq0}$, the matrices $H_k$
have different expressions for each $k$, i.e.,
\begin{equation}
\begin{aligned}
H_{-2-\ell} &= \,\myvcenter{\figname{Hmk}\mytikzHtmk{$l$}{$(\TQL)^\ell$}}\ , &
            &  \qquad \ell = 0,1,\ldots, \\[5pt]
H_{-1}      &= \,\myvcenter{\figname{Hm1}\mytikzHtmone{$Q$}{$\Sigma$}}\ ,\\[5pt]
H_0         &= \,\myvcenter{\figname{H0}\mytikzHtzero{$Q$}{$\Sigma$}}\ , \\[5pt]
H_{1+\ell}  &= \,\myvcenter{\figname{Hk}\mytikzHtk{$(\TQR)^\ell$}{$r$}}\ , &
            &  \qquad \ell = 0,1,\ldots,
\end{aligned}
\label{eq:Ht}
\end{equation}
with
\begin{align}
\myvcenter{\figname{iTR1c-HL-def}\mytikzHLdef{$l$}}
 &:= \,\myvcenter{\figname{iTR1c-hL}\mytikzhL{$\Sigma$}{$Q$}{$\Sigma$}{$Q$}}\,,&
\myvcenter{\figname{iTR1c-HR-def}\mytikzHRdef{$r$}}\,
 &:= \,\myvcenter{\figname{iTR1c-hR}\mytikzhR{$Q$}{$\Sigma$}{$Q$}{$\Sigma$}}\ .
\label{eq:iTRc-lr}
\end{align}
Note that, because the transfer matrices $\TQL$ and $\TQR$ have the dominant
eigenvalue 1, the infinite sums
\begin{equation}
\sum_{\ell=0}^{+\infty} (\TQL)^\ell, \qquad \mathrm{and} \qquad
\sum_{\ell=0}^{+\infty} (\TQR)^\ell
\label{eq:infsum-TQL-TQR}
\end{equation}
diverge, hence, the infinite sum in \eqref{eq:Htsum} also diverges.

On the other hand, we can show that for every $H_k$ in \eqref{eq:Ht},
\[
\vecz(Q_c)\T H_k \vecz(Q_c) = \theta,
\]
where $\theta$ is the the Rayleigh quotient \eqref{eq:iTRc-rq} and
$Q_\can$ the canonical center core \eqref{eq:bfx-frame-center-core}.
Therefore, we define the iTR residual as follows
\begin{equation}
\res := \sum_{k=-\infty}^{+\infty} \Big( H_k \vecz(Q_\can) -
                                         \theta \vecz(Q_\can) \Big)
     =: \sum_{k=-\infty}^{+\infty} \res_k.
\label{eq:iTRc-res}
\end{equation}
The following theorem shows how to construct this residual and
proves that, in contrast to the infinite sum in \eqref{eq:Htsum},
the iTR residual defined in \eqref{eq:iTRc-res} does converge.

\begin{theorem}[Residual]
\label{th:iTRc-res}
Let $\bfx$ be a normalized nonzero canonical iTR as in \defref{def:iTR-can}.
Then the residual for a given infinite-dimensional matrix $\bH$
\eqref{eq:H} is given by
\begin{equation}
\begin{aligned}
\res =\
 & \figname{iTR1c-res1}\mytikzresA{$Q_\can$}{$L$}\, +
 \,\figname{iTR1c-res2}\mytikzresB{$Q$}{$Q$}{$\Sigma$}{$\Sigma$}\  + \\[5pt]
 & \figname{iTR1c-res3}\mytikzresC{$Q$}{$Q$}{$\Sigma$}{$\Sigma$}\, +
 \,\figname{iTR1c-res4}\mytikzresD{$Q_\can$}{$R$}\,
 -\,4\theta\ \figname{vecQc}\mytikzveccore{$Q_\can$}\ ,
\end{aligned}
\label{eq:iTRc-res-th}
\end{equation}
where $\theta$ is the Rayleigh quotient \eqref{eq:iTRc-rq},
$Q_\can$ the canonical center core \eqref{eq:bfx-frame-center-core}, and
\begin{align}
\vecz(L)\T &:= \vecz(l)\T \left[\eye - \TtQL \right]\inv, &
\TtQL &:= \TQL - \vecz(\Sigma^2) \vecz(\eye)\T, \label{eq:iTRc-res-th-HL} \\
\vecz(R) &:= \left[\eye - \TtQR \right]\inv \vecz(r), &
\TtQR &:= \TQR - \vecz(\eye) \vecz(\Sigma^2)\T. \label{eq:iTRc-res-th-HR}
\end{align}
with $l$ and $r$ are defined by the diagrams given in \eqref{eq:Ht}.
\end{theorem}
\begin{proof}
Since each term in the infinite sum of \eqref{eq:iTRc-res} is defined
as the difference $\res_k = H_k \vecz(Q_\can) - \theta \vecz(Q_\can)$,
we immediately obtain from
\eqref{eq:Ht} that the 2nd and 3rd term in \eqref{eq:iTRc-res-th} are
equal to $H_{-1} \vecz(Q_\can)$ and $H_0 \vecz(Q_\can)$, respectively.
Next, we will prove that $\sum_{k=-\infty}^{-2} \res_k$ yields the 1st term
in \eqref{eq:iTRc-res-th} minus $\theta \vecz(Q_\can)$, and
$\sum_{k=1}^{+\infty} \res_k$ yields the 4th term minus $\theta \vecz(Q_\can)$.
We show that both infinite sums converge.

By using \eqref{eq:iTRc-rq} and \eqref{eq:Ht}, we diagrammatically
rewrite $\theta \vecz(Q_c)$ in the following factored form
\[
\theta\ \figname{vecQc}\mytikzveccore{$Q_\can$}\,
\fignames{iTR1c-res-proof-theta-Qc-}%
 = \,\mytikzresproofthetaQc{$Q$}{$Q$}{$\Sigma$}{$\Sigma$}{$Q_\can$}\,
 = \,\mytikzHLdefsmall{$l$}\!\!\!
   \underbrace{\mytikznotvecSvecIT{$\Sigma$}}_{\vecz(\Sigma^2) \vecz(\eye)\T}
   \!\!\!\mytikzresproofQcl{$Q_\can$}\ ,
\]
yielding
\begin{equation}
\res_{-2-\ell} = \,\figname{iTR1c-res-proof-Rmk}%
                 \mytikzresproofRl{$l$}{$(\TQL)^\ell$}{$Q_\can$}\,-\,
                 \theta\ \figname{vecQc}\mytikzveccore{$Q_\can$}\,
               = \,\figname{iTR1c-res-proof-Rtmk}%
                 \mytikzresproofRl{$l$}{$(\TtQL)^\ell$}{$Q_\can$}\ ,
\label{eq:iTRc-res-Rml}
\end{equation}
where $(\TtQL)^\ell = (\TQL)^\ell - \vecz(\Sigma^2) \vecz(\eye)\T$
for $\ell = 1,2,\ldots$.
Because $\vecz(\eye)$ and $\vecz(\Sigma^2)$ are the left and right
eigenvectors associated with the dominant eigenvalue 1 of $\TQL$,
as shown in \lemref{lem:can-tf-matrix-eigv}, the largest eigenvalue
of the deflated matrix $\TtQL$ is less than 1 (in magnitude).
Hence, in contrast to \eqref{eq:infsum-TQL-TQR},
the geometric series of $\TtQL$ converges, yielding
\[
\sum_{\ell=0}^{+\infty} (\TtQL)^\ell = \left[\eye - \TtQL\right]\inv.
\]
Note that this sum starts from 0 and \eqref{eq:iTRc-res-Rml} starts from 1.
Consequently, the infinite sum $\sum_{k=-\infty}^{-2} \res_k$ yields the
1st term in \eqref{eq:iTRc-res-th} minus $\theta \vecz(Q_\can)$,
where $L$ is defined by \eqref{eq:iTRc-res-th-HL}.

For the remaining part of the proof, we rewrite $\vecz(Q_\can) \theta$ as
\[
\figname{vecQc}\mytikzveccore{$Q_\can$}\ \theta\,
\fignames{iTR1c-res-proof-Qc-theta-}%
 = \,\mytikzresproofQctheta{$Q$}{$Q$}{$\Sigma$}{$\Sigma$}{$Q_\can$}\,
 = \,\mytikzresproofQcr{$Q_\can$}\!\!\!
   \underbrace{\mytikznotvecIvecST{$\Sigma$}}_{\vecz(\eye) \vecz(\Sigma^2)\T}
   \!\!\!\mytikzHRdefsmall{$r$}\ ,
\]
yielding
\begin{equation}
\res_{1+\ell} = \,\figname{iTR1c-res-proof-Rk}%
                \mytikzresproofRr{$Q_\can$}{$(\TQR)^\ell$}{$r$}\,-\,
                \figname{vecQc}\mytikzveccore{$Q_\can$}\ \theta\,
              = \,\figname{iTR1c-res-proof-Rtk}%
                \mytikzresproofRr{$Q_\can$}{$(\TtQR)^\ell$}{$r$}\ ,
\label{eq:iTRc-res-Rl}
\end{equation}
for $\ell = 1,2,\ldots$.
Using similar arguments as before, we can show that the geometric series of
$\TtQR$ converges and that the infinite sum $\sum_{k=1}^{+\infty} \res_k$ yields
the 4th term in \eqref{eq:iTRc-res-th} minus $\theta \vecz(Q_\can)$.
Remark that the term $\theta \vecz(Q_\can)$ in $\res_k$ always gets incorporated
in $\TtQL$ or $\TtQR$, except for $k = -2,-1,0,1$, yielding the 5th term in
\eqref{eq:iTRc-res-th}.
This completes the proof.
\end{proof}

In order to define the residual for a Ritz pair $(\theta,\bfx)$ with $\bfx$
being an iTR2, we first introduce the frame matrices $\bF_{\neq Q}$ and
$\bF_{\neq U}$ which allow us to rewrite $\bfx$ as follows
\begin{equation}
\bfx = \bF_{\neq Q} \vecz(Q_\can) = \bF_{\neq U} \vecz(U_\can),
\label{eq:bfx-frame2}
\end{equation}
where $\bF_{\neq Q}\T \bF_{\neq Q} = \bF_{\neq U}\T \bF_{\neq U} = \eye$ and
the \emph{canonical center cores} $Q_\can$ and $U_\can$ are
\fignames{iTR2c-center-core-}%
\begin{align}
\mytikzcore{$Q_\can$} &= \mytikzcentercore{$\Omega$}{$Q$}{$\Sigma$}\,,&
\mytikzcore{$U_\can$} &= \mytikzcentercore{$\Sigma$}{$U$}{$\Omega$}\,.
\label{eq:iTRc-center-cores}
\end{align}
respectively.
Next, we define the residual in an average sense as follows
\begin{align*}
\res =
 &\ \frac{1}{2} \sum_{k=-\infty}^{+\infty} \Big(
    \bF_{\neq Q}\T \bH_k \bF_{\neq Q} \vecz(Q_\can) - \theta \vecz(Q_\can)
    \Big)\ + \\
 &\ \frac{1}{2} \sum_{k=-\infty}^{+\infty} \Big(
    \bF_{\neq U}\T \bH_k \bF_{\neq U} \vecz(U_\can) - \theta \vecz(U_\can)
    \Big),
\end{align*}
where $\theta$ is the Rayleigh quotient \eqref{eq:iTR2c-rq}, and $Q_\can$
and $U_\can$ the canonical center cores \eqref{eq:iTRc-center-cores}.
The following theorem summarizes how to compute this residual.

\begin{theorem}[iTR2 residual]
\label{th:iTR2c-res}
Let $\bfx$ be a normalized nonzero canonical iTR2 as in \defref{def:iTR-can2}.
Then the residual for a given infinite-dimensional matrix $\bH$
\eqref{eq:H} is given by
\begin{equation}
\begin{aligned}
\res =\
 &\frac{1}{2}\,
  \figname{iTR2c-resQ1}\mytikzresA{$Q_\can$}{$L_Q$}\, +
  \frac{1}{2}\,
  \figname{iTR2c-resQ2}\mytikzresB{$U$}{$Q$}{$\Omega$}{$\Sigma$}\ + \\
 &\frac{1}{2}\,
  \figname{iTR2c-resQ3}\mytikzresC{$Q$}{$U$}{$\Sigma$}{$\Omega$}\, +
  \frac{1}{2}\,
  \figname{iTR2c-resQ4}\mytikzresD{$Q_\can$}{$R_Q$}\,
    -\,3\theta\ \figname{vecQc}\mytikzveccore{$Q_\can$}\ + \\
 &\frac{1}{2}\,
  \figname{iTR2c-resU1}\mytikzresA{$U_\can$}{$L_U$}\, +
  \frac{1}{2}\,
  \figname{iTR2c-resU2}\mytikzresB{$Q$}{$U$}{$\Sigma$}{$\Omega$}\ + \\
 &\frac{1}{2}\,
  \figname{iTR2c-resU3}\mytikzresC{$U$}{$Q$}{$\Omega$}{$\Sigma$}\, +
  \frac{1}{2}\,
  \figname{iTR2c-resU4}\mytikzresD{$U_\can$}{$R_U$}\,
    -\,3\theta\ \figname{vecUc}\mytikzveccore{$U_\can$}\ ,
\end{aligned}
\label{eq:iTR2c-res-th}
\end{equation}
where $\theta$ is the Rayleigh quotient \eqref{eq:iTR2c-rq},
$Q_\can$ and $U_\can$ the canonical center cores \eqref{eq:iTRc-center-cores},
\begin{equation}
\begin{aligned}
\vecz(L_Q)\T &:= \left[ \vecz(l_Q)\T + \vecz(l_U)\T\TUL \right]
                 \left[ \eye - \TtQUL                   \right]\inv, \\
\vecz(L_U)\T &:= \left[ \vecz(l_U)\T + \vecz(l_Q)\T\TQL \right]
                 \left[ \eye - \TtUQL                   \right]\inv, \\
\vecz(R_Q) &:= \left[ \eye - \TtUQR               \right]\inv
               \left[ \vecz(r_U) + \TUR\vecz(r_Q) \right], \\
\vecz(R_U) &:= \left[ \eye - \TtQUR               \right]\inv
               \left[ \vecz(r_Q) + \TQR\vecz(r_U) \right],
\end{aligned}
\label{eq:iTR2c-res-linsolve}
\end{equation}
with
\begin{align*}
\TtQUL &:= \TQUL - \vecz(\Omega^2) \vecz(\eye)\T, &
\TtQUR &:= \TQUR - \vecz(\eye) \vecz(\Omega^2)\T, \\
\TtUQL &:= \TUQL - \vecz(\Sigma^2) \vecz(\eye)\T, &
\TtUQR &:= \TUQR - \vecz(\eye) \vecz(\Sigma^2)\T,
\end{align*}
and
\begingroup
\allowdisplaybreaks
\begin{align*}
\myvcenter{\figname{iTR2c-HLQ-def}\mytikzHLdef{$l_Q$}}   &:= \,
\myvcenter{\figname{iTR2c-hLQ}\mytikzhL{$\Omega$}{$Q$}{$\Sigma$}{$U$}}\,, &
\myvcenter{\figname{iTR2c-HRQ-def}\mytikzHRdef{$r_Q$}}\, &:= \,
\myvcenter{\figname{iTR2c-hRQ}\mytikzhR{$Q$}{$\Sigma$}{$U$}{$\Omega$}}\ ,\\[5pt]
\myvcenter{\figname{iTR2c-HLU-def}\mytikzHLdef{$l_U$}}   &:= \,
\myvcenter{\figname{iTR2c-hLU}\mytikzhL{$\Sigma$}{$U$}{$\Omega$}{$Q$}}\,, &
\myvcenter{\figname{iTR2c-HRU-def}\mytikzHRdef{$r_U$}}\, &:= \,
\myvcenter{\figname{iTR2c-hRU}\mytikzhR{$U$}{$\Omega$}{$Q$}{$\Sigma$}}\ .
\end{align*}
\endgroup
\end{theorem}
\begin{proof}
The proof is similar to the one for \thref{th:iTRc-res}.
\end{proof}

\section{Power method}
\label{sec:fpm}

One way to compute the algebraically smallest eigenvalue $\lambda_1$
of a symmetric positive definite matrix $\bH$ is to apply the power method to the matrix exponential $e^{-\bH}$.
If $\lambda_1$ is simple, the power method converges linearly to the desired
eigenpair $(\lambda_1, \bfx)$ at the rate $e^{\lambda_2 - \lambda_1}$.
This approach is generally not recommended because working with the matrix
exponential $e^{-\bH}$ can be costly.
In fact, computing $e^{-\bH}$ or applying $e^{-\bH}$ to a vector may be harder
than computing selected eigenvalues of $\bH$.
However, when $\bH$ has the structure exhibited in \eqref{eq:H},
$e^{-\bH t}$ may be approximated by a simpler form that makes it
possible to (approximately) perform a power iteration for $e^{-\bH t}$ if $t$ is sufficiently small.

Throughout this section, we will make use of iTR matrices, which are graphically
represented as follows
\[
\bA = \,\vcenter{\hbox{\figname{iTRmatrix}\mytikziTRmatrix{$A$}{-1}{0}{1}}}\ ,
\]
where all indices $i_k,j_k$ run from 1 to $d$ and each $A(j,i)$ is a matrix of
size $r \times r$, with $r$ referred to as the \emph{rank} of $\bA$.

\subsection{Approximating $\exp(-\bH t)$ by matrix splitting}

We now discuss how to approximate the operation $e^{-\bH t} \bfx$, with
$\bfx$ being an iTR, which is required in a single step of the power iteration. Because the terms 
$\bH_k$ in the sum $\bH = \sum_k \bH_k$ do not commute,  the expressions
$e^{-\bH t}$ and 
$\cdots e^{-\bH_{k-1}} e^{-\bH_k} e^{-\bH_{k+1}} \cdots$ are not the same but they can be expected to be close for small $t$.
Indeed, the Lie product formula, 
also known as Suzuki--Trotter splitting, states that
\[
e^{(A + B)t} = e^{At} e^{Bt}  + \bigO(t^2),
\]
holds for square matrices $A$ and $B$. This suggest the following approximation:
\begin{equation}
e^{-\bH t} \approx \prod_{k=-\infty}^{+\infty} e^{-\bH_k t}, \qquad t \approx 0.
\label{eq:stsplit}
\end{equation}
In the physics literature, $t$ is viewed as an imaginary time variable and,
hence, $e^{-\bH t}$ is often referred to as imaginary time evolution.

Let us first assume that the matrices $\bH_k$ only act on one site, that is,
\[
\bH_k = \cdots \otimes \eye[d] \otimes \eye[d] \otimes A_k
               \otimes \eye[d] \otimes \eye[d] \otimes \cdots,
\qquad A_k = A \inRR{d}.
\]
By making use of the identity
$
e^{\eye \otimes A \otimes \eye} = \eye \otimes e^{A} \otimes \eye$,
the matrix exponentials $e^{-\bH_k t}$ can be simplified as follows
\[
e^{-\bH_k t} = \cdots \otimes \eye \otimes \eye \otimes e^{-A t}
                      \otimes \eye \otimes \eye \otimes \cdots.
\]
In turn, the infinite product in \eqref{eq:stsplit} is the rank-1 iTR matrix
\begin{align*}
\prod_{k=-\infty}^{+\infty} e^{-\bH_k t}
 &= \cdots \otimes e^{-A t} \otimes e^{-A t} \otimes e^{-A t} \otimes \cdots
  = \cdots \figname{iTRmatrix-exp-At}\mytikzexpAt{$e^{-At}$}{-1}{0}{1} \cdots,
\end{align*}
where the dotted lines correspond to rank $r = 1$.
Hence, if $\bfx$ is an iTR, the product
\begin{equation}
\bfy = \left( \prod_{k=-\infty}^{+\infty} e^{-\bH_k t} \right) \bfx,
\label{eq:splitk}
\end{equation}
is again an iTR of the same rank, and \eqref{eq:splitk} can be efficiently 
computed by left multiplying the slices of $\bfx$ with the $d \times d$
matrix $e^{-A t}$
\[
\fignames{iTRcore-YexpAtX-}%
\mytikzcore{$Y$} = \mytikzYexpAtX{$X$}{$e^{-At}$}\,,
\]
where $X$ and $Y$ are the cores of $\bfx$ and $\bfy$, respectively.

Let us now consider the case that $\bH_k$ acts on two neighbors,
that is,
\begin{equation}
\bH_k = \cdots \otimes \eye[d] \otimes \eye[d] \otimes M_{k,k+1}
               \otimes \eye[d] \otimes \eye[d] \otimes \cdots,
\qquad M_{k,k+1} = M \inRR{d^2}.
\label{eq:Hk-neighbor}
\end{equation}
Then $\bH_k$ will interact with the $k$th and $(k+1)$st cores of $\bfx$.
Because $\bH_k$ and $\bH_{k+1}$ both interact with the $(k+1)$st core of $\bfx$,
$e^{-\bH t} \bfx$ cannot be directly obtained as in \eqref{eq:splitk}.
Therefore, we first reorder the summation and group the
even and odd terms together:
\[
\bH = \left( \sum_{k=-\infty}^{+\infty} \bH_{2k} \right) +
      \left( \sum_{k=-\infty}^{+\infty} \bH_{2k+1} \right)
    =: \ \bH_\even + \bH_\odd.
\]
Applying the Suzuki--Trotter splitting to both terms, yields
\begin{equation}
e^{-\bH t} = e^{-(\bH_\even + \bH_\odd)t}
 \approx \left( \prod_{k=-\infty}^{+\infty} e^{-\bH_{2k} t} \right)
         \left( \prod_{k=-\infty}^{+\infty} e^{-\bH_{2k+1} t} \right),
\label{eq:H-eo-split}
\end{equation}
for sufficiently small $t$, where
\begin{align}
\prod_{k=-\infty}^{+\infty} e^{-\bH_{2k} t}
 &= \figname{iTRmatrix-exp-Mt-even}%
\mytikzexpMt{$e^{-Mt}$}{2k-2}{2k-1}{2k}{2k+1}{2k+2}{2k+3}\,,
\label{eq:prod-Hk-even} \\[5pt]
\prod_{k=-\infty}^{+\infty} e^{-\bH_{2k+1} t}
 &= \hspace{10mm} \figname{iTRmatrix-exp-Mt-odd}%
\mytikzexpMt{$e^{-Mt}$}{2k-1}{2k}{2k+1}{2k+2}{2k+3}{2k+4}\,.
\label{eq:prod-Hk-odd}
\end{align}
Remark that both \eqref{eq:prod-Hk-even} and \eqref{eq:prod-Hk-odd} can also be considered
as local operators when we combine the corresponding neighboring indices.
In particular, each of the even and odd infinite products in \eqref{eq:H-eo-split}
result in rank-1 iTR matrices
\begin{align}
\prod_{k=-\infty}^{+\infty} e^{-\bH_{2k} t}
 &= \figname{iTRmatrix-exp-Mt-even-comb}%
\mytikzexpMtcomb{$e^{-Mt}$}{2k-2,2k-1}{2k,2k+1}{2k+2,2k+3}\,,
\label{eq:prod-Hk-even-combined} \\[5pt]
\prod_{k=-\infty}^{+\infty} e^{-\bH_{2k+1} t}
 &= \hspace{10mm} \figname{iTRmatrix-exp-Mt-odd-comb}%
\mytikzexpMtcomb{$e^{-Mt}$}{2k-1,2k}{2k+1,2k+2}{2k+3,2k+4}\,,
\label{eq:prod-Hk-odd-combined}
\end{align}
where the neighboring indices
$i_{2k,2k+1} = (i_{2k},i_{2k+1})$ and
$i_{2k+1,2k+2} = (i_{2k+1},i_{2k+2})$ are combined, respectively.

In order to approximate $e^{-\bH t} \bfx$,
we first focus on the odd terms, i.e.,
\begin{equation}
\bfy_\odd = \left( \prod_{k=-\infty}^{+\infty} e^{-\bH_{2k+1} t} \right) \bfx,
\label{eq:split-yo}
\end{equation}
which can be computed in the same way as \eqref{eq:splitk} by reformulating
$\bfx$ as an iTR2 with 2 cores $X$, i.e., combining the $(2k+1)$st and
$(2k+2)$nd cores into a \emph{supercore} $\cX$ of dimension
$r \times d^2 \times r$.
Hence, \eqref{eq:split-yo} only requires the left multiplication of the slices
of the new supercore with the $d^2 \times d^2$ matrix $e^{-Mt}$
\[
\fignames{iTRsupercore-expMt-supercore-odd-}%
\mytikzsupercoredim{$\cY_\odd$}{$d^2$} = \mytikzYexpMtX{$\cX$}{$e^{-Mt}$}\,,
\qquad\quad \mbox{with} \qquad\quad
\mytikzsupercoredim{$\cX$}{$d^2$}
 = \mytikzsupercoresplitteddim{$X$}{$X$}{$d$}{$d$}\,,
\]
where $d$ and $d^2$ next to the crossed out legs correspond to the size of the (super)cores.
Thereafter, we only need to multiply the even terms. However, directly applying
\begin{equation}
\bfy
 = \left( \prod_{k=-\infty}^{+\infty} e^{-\bH_{2k} t} \right) \bfy_\odd
\label{eq:split2kp1}
\end{equation}
is not possible because the supercores $\cY_\odd$ are formed by combining the
indices $(2k+1,2k+2)$, while \eqref{eq:prod-Hk-even-combined} requires the
indices $(2k,2k+1)$ to be combined.
Therefore, we first need to split up the supercore of $\bfy_\odd$ back into 2 normal cores,
followed by a new and different recombination of the cores in order to carry
out \eqref{eq:split2kp1} efficiently.

Note that the size of $\cY_\odd$ is $r \times d^2 \times r$ and in order to split it up
into 2 cores of size $r \times d \times r$, we necessarily need to perform an approximation.
A possible way to achieve this is by first reshaping $\cY_\odd$ into an
$rd \times rd$ matrix $C$, next approximating $C$ by a rank-$r$ factorization
\[
C \approx C_1 C_2\T,
\]
where $C_1, C_2 \inR[dr][r]$,
and finally reshaping the factors $C_1$ and $C_2$ into 3rd-order tensors
\[
\fignames{iTR2-split-odd-}%
\mytikziTRsplitA{$\cY_\odd$}{$r$}{$d^2$} =
\mytikziTRsplitB{$C$}{$dr$}{$rd$} \approx
\mytikziTRsplitC{$C_1$}{$C_2\T$}{$dr$}{$r$}{$rd$} =
\mytikziTRsplitD{$Y_{1\odd}$}{$Y_{2\odd}$}{$r$}{$d$}\ .
\]
In order words, $\bfy_\odd$ is approximated by
\[
\bfyt_\odd
 = \,\figname{iTR2-bfyt-odd-approx}\mytikziTRRapprox{$Y_{1\odd}$}{$Y_{2\odd}$}%
{$i_{2k-1}$}{$i_{2k}$}{$i_{2k+1}$}{$i_{2k+2}$}\ .
\]
Consequently, we can now approximate \eqref{eq:split2kp1}
by replacing $\bfy_\odd$ with $\bfyt_\odd$
\[
\fignames{iTRsupercore-expMt-supercore-}%
\mytikzsupercoredim{$\cY$}{$d^2$}
 = \mytikzYexpMtX{$\widetilde\cY_\odd$}{$e^{-Mt}$}\,,
\qquad\quad \mbox{with} \qquad\quad
\mytikzsupercoredim{$\widetilde\cY_\odd$}{$d^2$}
 = \mytikzsupercoresplitteddim{$Y_{2\odd}$}{$Y_{1\odd}$}{$d$}{$d$}\,.
\]
In the final step, we split the supercore $\cY$ into 2 cores $Y_2$ and $Y_1$
\[
\fignames{iTR2-split-}%
\mytikziTRsplitA{$\cY$}{$r$}{$d^2$} =
\mytikziTRsplitB{$D$}{$dr$}{$rd$} \approx
\mytikziTRsplitC{$D_1$}{$D_2\T$}{$dr$}{$r$}{$rd$} =
\mytikziTRsplitD{$Y_2$}{$Y_1$}{$r$}{$d$}\ .
\]
or in order words, $\bfy$ is approximated by
\[
\bfyt
 = \,\figname{iTR2-bfyt-approx}\mytikziTRRapprox{$Y_2$}{$Y_1$}%
{$i_{2k}$}{$i_{2k+1}$}{$i_{2k+2}$}{$i_{2k+3}$}\ .
\]
Remark that, although we started with an iTR $\bfx$, the result of applying
$e^{-\bH t}$ to it is an iTR2.
Therefore, in \figref{fig:matvec} we graphically illustrate the application of
\eqref{eq:stsplit} to an iTR2 with cores $X_1$ and $X_2$.
Due to the translational invariance, we only need to apply $e^{-Mt}$ twice in
order to compute $\bfyt$, which is indicated in black.
The computation indicated in gray can completely be ignored, so that the
multiplication $\bfy = e^{-\bH t}\bfx$ can be computed in an efficient way.

\begin{figure}[hbtp]
\begin{align*}
\left( \prod_{k=-\infty}^{+\infty} e^{-\bH_{2k+1} t} \right) \bfx \,
 &= \,\figname{iTR2-matvecA}%
\mytikzmatvecA{$X_1$}{$X_2$}{$e^{-Mt}$} \\[5pt]
\left( \prod_{k=-\infty}^{+\infty} e^{-\bH_{2k} t} \right) \bfyt_\odd \,
 &= \,\figname{iTR2-matvecB}%
\mytikzmatvecB{$Y_{1\odd}$}{$Y_{2\odd}$}{$e^{-Mt}$} \\[5pt]
\bfyt \,
 &= \,\figname{iTR2-matvecC}\mytikzmatvecC{$Y_1$}{$Y_2$}
\end{align*}
\vspace{-10pt}
\caption{The matrix--vector operation $\bfy = e^{-\bH t}\bfx$.}
\label{fig:matvec}
\end{figure}

\subsection{Algorithm}
\label{sec:alg}

As illustrated in \figref{fig:matvec}, we only need to operate on 1 supercore
at a time.
Instead of working with supercores involving $X_1,X_2$ and
$Y_{2\odd},Y_{1\odd}$, respectively,
we will operate on the following canonical center supercores $\cX_{1\can}$
and $\cX_{2\can}$:
\[
\fignames{iTR2c-center-coreQU-}%
\mytikzsupercore{$\cX_{1\can}$}
 = \mytikzcentersupercore{$Q$}{$U$}{$\Sigma$}{$\Omega$}\,,
\quad \mathrm{and} \quad
\fignames{iTR2c-center-coreUQ-}%
\mytikzsupercore{$\cX_{2\can}$}
 = \mytikzcentersupercore{$U$}{$Q$}{$\Omega$}{$\Sigma$}\,,
\]
respectively, so that their corresponding frame matrices are orthogonal.

We only describe the application of the odd part of $e^{-\bH t}$ onto
$\cX_\can$, since the even part can be performed in exactly the same way
by interchanging $Q,U$ and $\Sigma,\Omega$, respectively.
The computation consists of 5 steps as illustrated in \Figref{fig:matvec-step}.
In step~1, the canonical center supercore $\cX_{1\can}$ is formed.
Next, we apply $e^{-Mt}$ to it in step~2, followed by reshaping the resulting
supercore $\cY_{1\can}$ into a matrix.
In step~3, we approximate the matrix representation of $\cY_{1\can}$ by a
truncated SVD, with the diagonal matrix $S_r$ containing only the $r$ largest
singular values, followed by reshaping the left and right singular vectors
$W$ and $V$ into 3rd-order tensors, i.e., cores.
Then in step~4, in order to update the cores $Q$ and $U$ by $W_\Omega$ and
$V_\Omega\T$, respectively, and the diagonal matrix $\Sigma$ by $S_r$, we
transform the result of step~3 into the form of step~1 by left and right
multiplying $W$ and $V\T$, respectively, by $\Omega\inv$.
Note that, although the resulting form resembles the form of step~1,
the former is not in canonical form yet.
Therefore, we update in step~5 both cores $Q_\star$ and $U_\star$, and both
diagonal matrices $\Sigma_\star$ and $\Omega_\star$ to bring the updated iTR2
again in canonical form.

\begin{figure}[hbtp]
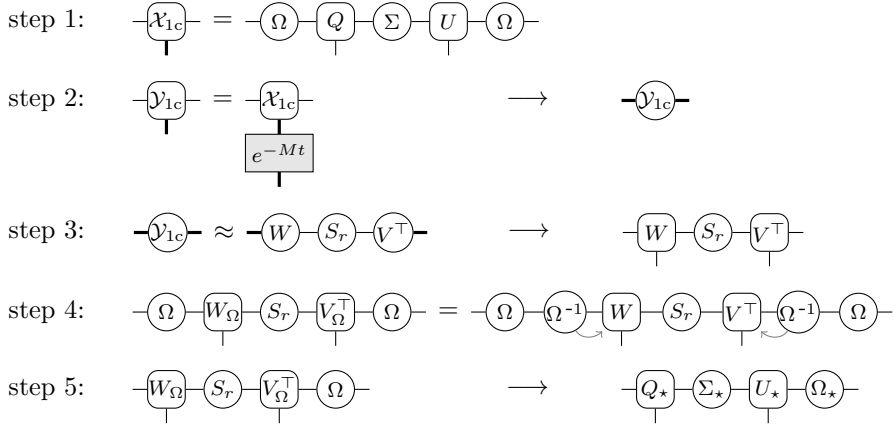

\begin{align*}
& \mbox{step 1:} && \figname{iTR2-matvec-step1}%
\mytikzmatvecstepA{$\cX_{1\can}$}{$Q$}{$U$}{$\Sigma$}{$\Omega$} \\[5pt]
& \mbox{step 2:} && \figname{iTR2-matvec-step2}%
\mytikzmatvecstepB{$\cY_{1\can}$}{$\cX_{1\can}$}{$e^{-Mt}$} \\[5pt]
& \mbox{step 3:} && \figname{iTR2-matvec-step3}%
\mytikzmatvecstepC{$\cY_{1\can}$}{$W$}{$S_r$}{$V\T$} \\[5pt]
& \mbox{step 4:} && \figname{iTR2-matvec-step4}%
\mytikzmatvecstepD{$W_\Omega$}{$V_\Omega\T$}{$S_r$}{$\Omega$}%
{$\Omega^{\,\mbox{-}1}$}{$W$}{$S_r$}{$V\T$} \\[5pt]
& \mbox{step 5:} && \figname{iTR2-matvec-step5}%
\mytikzmatvecstepE{$W_\Omega$}{$V_\Omega\T$}{$S_r$}{$\Omega$}%
{$Q_\star$}{$U_\star$}{$\Sigma_\star$}{$\Omega_\star$}
\end{align*}
\vspace{-10pt}
\caption{The odd part of the matrix--vector operation $\bfy = e^{-\bH t}\bfx$,
where $\bfx$ is given in canonical iTR2 form.
Note that the even part corresponds to interchanging the cores $Q$ and $U$,
and the diagonal matrices $\Sigma$ and $\Omega$, respectively.
(1) Forming the canonical center supercore.
(2) Applying the matrix exponential followed by reshaping the result into
    a matrix.
(3) Performing a truncated SVD and reshaping the singular vectors into cores.
(4) Updating the cores and the central diagonal matrix, while keeping $\Omega$
    unchanged.
(5) Computing the canonical iTR2 form.}
\label{fig:matvec-step}
\end{figure}

We now have all components to state the method for computing the smallest
eigenvalue of \eqref{eq:H}.
Due to the approximation made by the Suzuki--Trotter splitting of $e^{-\bH t}$
and the rank truncation for splitting a supercore into 2 normal cores,
the multiplication $\bfy = e^{-\bH t}\bfx$ is inexact.
Even when $t$ is fixed, the error introduced by the rank truncation is different
in each iteration and results in the application of a slightly different
operator every iteration.
\algref{alg:fpi-can} shows the outline of the power method for
computing the smallest eigenvalue of \eqref{eq:H}.
This algorithm takes a canonical iTR2 as starting vector and returns the
Rayleigh quotient and its corresponding eigenvector in canonical iTR2 form.
In every iteration of \algref{alg:fpi-can}, we apply the matrix exponential
$e^{-Mt}$ twice, according to \figref{fig:matvec-step}, followed by the
computation of the Rayleigh quotient and residual.

\begin{algorithm2e}[hbtp]
\caption{iTR power method}%
\label{alg:fpi-can}%
\SetKwInOut{input}{Input}
\SetKwInOut{output}{Output}
\BlankLine
\input{Canonical iTR2 with cores $Q,U \inR^{r \times d \times r}$ and $\Sigma,\Omega \inRR{r}$\\
Nearest neighbor interaction matrix $M$ and initial timestep $t$}
\output{Rayleigh quotient and corresponding eigenvector in canonical iTR2}
\BlankLine
\While {not converged} {
\nl  Apply $e^{-Mt}$ to supercore $\cX_{1\can}$
     according to \figref{fig:matvec-step}. \\
\nl  Apply $e^{-Mt}$ to supercore $\cX_{2\can}$
     according to \figref{fig:matvec-step}. \\
\nl  Compute Rayleigh quotient $\theta$ as in \eqref{eq:iTR2c-rq}. \\
\nl  Check convergence based on residual \eqref{eq:iTR2c-res-th}. \\
\nl  Update $t$.
}
\end{algorithm2e}

As we will show in the numerical experiments, the time step $t$ can have a
significant effect on the convergence and accuracy of the Rayleigh quotient.
On the one hand, when $t$ is chosen very small then the eigenvalues of 
$e^{-\bH t}$ are close to $1$ and, in turn, the gap between the smallest and the second smallest eigenvalue becomes small, leading to a 
high number of iterations.
On the other hand, if the time step is chosen too large, the convergence of
\algref{alg:fpi-can} stagnates prematurely due to the dominating Suzuki--Trotter
splitting error.
It is therefore recommended to adapt and decrease the time step $t$ in the
course of the algorithm.
By starting with a relative large $t$, the power method will stagnate in a
small number of iterations.
We can use the iTR2 residual for detecting stagnation.
Next, we decrease $t$ and continue the power method with the already
obtained approximate eigenvector.
By updating $t$ and using more accurate approximate eigenvectors for
smaller $t$, we will be able to reduce the total number of iterations.

\subsection{Computational considerations}

Most operations of \algref{alg:fpi-can} can be efficiently implemented
involving only 3rd-order tensors of size $r \times d \times r$, such as
steps~1--4 in \figref{fig:matvec-step} and the Rayleigh quotient.
On the other hand, step~5 in \figref{fig:matvec-step} and the computation
of the residual involve transfer matrices of size $r^2 \times r^2$.
The former requires computing dominant eigenvectors of the left and right
transfer matrices and the latter needs linear systems solves involving the
same matrices.
In practice, we can use the relation~\eqref{eq:vectrick} 
to avoid forming these $r^2 \times r^2$ matrices explicitly and rewrite
everything in terms of only $r \times r$ matrices.
However, even for low values of $r \approx 10$, more than 90\% of the wall clock
time is still spent on these operations involving the transfer matrices.
Therefore, we will make use of iterative methods for computing the dominant
eigenvectors and the residual.
We could choose not to compute the residual in every iteration but the cost
for transforming the updated iTR2 to canonical form remains.

Therefore, we also propose a second variant of the power method
that only works with standard iTR2s, i.e., ignoring step~5 in
\figref{fig:matvec-step} such that the updated cores are $Q_\star := W_\Omega$
and $U_\star := V_\Omega\T$, and the updated diagonal matrices are 
$\Sigma_\star := S_r$ and $\Omega_\star := \Omega$.
Note that in this case only 1 of the 2 diagonal matrices is updated, hence
the resulting iTR2 will, in general, not satisfy all the orthogonality
conditions of \defref{def:iTR-can2}.
However, due to the alternating singular value decompositions on
$\cY_{1\can}$ and $\cY_{2\can}$, we observe that the resulting iTR2 converges
to canonical form.
As a result, we can still use the relatively cheap canonical Rayleigh
quotient \eqref{eq:iTR2c-rq} compared to the standard Rayleigh quotient
\eqref{eq:rq2} which would again require the computation of dominant
eigenvectors of the transfer matrices.

Note that in the iTEBD algorithm described in \cite{vida2007,orvi2008}
the matrix-vector operation $\bfy = e^{-\bH t} \bfx$ is implemented in
a slightly different way.
Instead of bringing the updated iTR2 to canonical form, as done in step~5 in
\figref{fig:matvec-step}, the iTR2 is directly transformed to canonical form
right after step~2 and before the truncation is applied in step~3.
Hence, the resulting iTR2 is only approximately in canonical form.

\newcommand{\myplotiter}[3][5]{%
\begin{tikzpicture}
\begin{loglogaxis}[
  width=0.48\textwidth,
  #3,
  legend pos=north east,
  legend style={draw=none,fill=none}
]
\ifnum#1>0
  \addplot[myStyOne] table[x index=1,y index=5] {\datfile{#2-tau=1e-01}};
  \addplot[myStyRes] table[x index=1,y index=4] {\datfile{#2-tau=1e-01}};
\fi
\ifnum#1>1
  \addplot[myStyTwo] table[x index=1,y index=5] {\datfile{#2-tau=1e-02}};
  \addplot[myStyRes] table[x index=1,y index=4] {\datfile{#2-tau=1e-02}};
\fi
\ifnum#1>2
  \addplot[myStyThr] table[x index=1,y index=5] {\datfile{#2-tau=1e-03}};
  \addplot[myStyRes] table[x index=1,y index=4] {\datfile{#2-tau=1e-03}};
\fi
\ifnum#1>3
  \addplot[myStyFou] table[x index=1,y index=5] {\datfile{#2-tau=1e-04}};
  \addplot[myStyRes] table[x index=1,y index=4] {\datfile{#2-tau=1e-04}};
\fi
\ifnum#1>4
  \addplot[myStyFiv] table[x index=1,y index=5] {\datfile{#2-tau=1e-05}};
  \addplot[myStyRes] table[x index=1,y index=4] {\datfile{#2-tau=1e-05}};
\fi
\ifnum#1>5
  \addplot[myStySix] table[x index=1,y index=5] {\datfile{#2-tau=1e-06}};
  \addplot[myStyRes] table[x index=1,y index=4] {\datfile{#2-tau=1e-06}};
\fi
\legend{,$\norm{R}$}
\end{loglogaxis}
\end{tikzpicture}%
}
\newcommand{\myplottime}[3][5]{%
\begin{tikzpicture}
\begin{semilogyaxis}[
  width=0.48\textwidth,
  #3,
  legend pos=north east,
  legend style={draw=none,fill=none}
]
\ifnum#1>0
  \addplot[myStyOne] table[x index=2,y index=5] {\datfile{#2-tau=1e-01}};
  \addplot[myStyRes] table[x index=2,y index=4] {\datfile{#2-tau=1e-01}};
\fi
\ifnum#1>1
  \addplot[myStyTwo] table[x index=2,y index=5] {\datfile{#2-tau=1e-02}};
  \addplot[myStyRes] table[x index=2,y index=4] {\datfile{#2-tau=1e-02}};
\fi
\ifnum#1>2
  \addplot[myStyThr] table[x index=2,y index=5] {\datfile{#2-tau=1e-03}};
  \addplot[myStyRes] table[x index=2,y index=4] {\datfile{#2-tau=1e-03}};
\fi
\ifnum#1>3
  \addplot[myStyFou] table[x index=2,y index=5] {\datfile{#2-tau=1e-04}};
  \addplot[myStyRes] table[x index=2,y index=4] {\datfile{#2-tau=1e-04}};
\fi
\ifnum#1>4
  \addplot[myStyFiv] table[x index=2,y index=5] {\datfile{#2-tau=1e-05}};
  \addplot[myStyRes] table[x index=2,y index=4] {\datfile{#2-tau=1e-05}};
\fi
\ifnum#1>5
  \addplot[myStySix] table[x index=2,y index=5] {\datfile{#2-tau=1e-06}};
  \addplot[myStyRes] table[x index=2,y index=4] {\datfile{#2-tau=1e-06}};
\fi
\legend{,$\norm{R}$}
\end{semilogyaxis}
\end{tikzpicture}%
}

\section{Numerical experiments}
\label{sec:exp}

In this section we illustrate the power method, outlined in
\algref{alg:fpi-can}, for 3 examples:
the one-dimensional Ising model with a transverse field,
the spin $S = 1$ Heisenberg isotropic antiferromagnetic model, and
the spin $S = 1/2$ Heisenberg model.

All numerical experiments are performed in MATLAB version 9.6.0 (R2019a) on
a MacBook Pro running an Intel(R) Core(TM) i7-5557U CPU @ 3.1 GHz
dual core processor with 16 GB RAM.
The dominant eigenvectors of the transfer matrices are computed via
\texttt{eigs} and the linear systems for computing the residual in
\eqref{eq:iTR2c-res-linsolve} are solved via \texttt{gmres} with
tolerance 1e$-8$.

\subsection{One-dimensional transverse field Ising model}

We start with the one-dimensional transverse field Ising model and
corresponding infinite-dimensional Hamiltonian
\begin{equation}
\bH = -\sum_{k=-\infty}^{+\infty}
  \cdots \otimes I \otimes \sigma_z \otimes \sigma_z \otimes I \otimes \cdots
- g \sum_{k=-\infty}^{+\infty}
  \cdots \otimes I \otimes \sigma_x \otimes I \otimes \cdots
\label{eq:TFI}
\end{equation}
where
$\sigma_x = \begin{bsmallmatrix} 0 & 1 \\ 1 & 0 \end{bsmallmatrix}$ and
$\sigma_z = \begin{bsmallmatrix} 1 & 0 \\ 0 &-1 \end{bsmallmatrix}$
are the 1st and 3rd Pauli matrices,
respectively, and the following $4 \times 4$ symmetric matrix
\[
M =
\begin{bmatrix}
 -1 & -g &  0 &  0 \\
 -g &  1 &  0 &  0 \\
  0 &  0 &  1 & -g \\
  0 &  0 & -g & -1
\end{bmatrix}
\]
as nearest neighbor interaction matrix in the notation of \eqref{eq:H}.
We choose the parameter $g = 2$.
The exact solution of the smallest eigenvalue of \eqref{eq:TFI} is
\cite{pfeu1970}
\begin{align*}
\lambda_0(g) &= -\frac{1}{2\pi} \int_{-\pi}^{\pi} \sqrt{1+g^2-2g\cos x}\,dx, &
\lambda_0(2) &\approx -2.127\,088\,819\,946\,730.
\end{align*}

In the first experiment, we compare the 2 variants of \algref{alg:fpi-can},
i.e., the canonical variant where we always convert the result of applying
$e^{-Mt}$ back into canonical iTR2 form and the non-canonical variant where
we skip step~5 in \figref{fig:matvec-step}.
In order to improve the numerical stability in the latter one, we normalize
the diagonal matrices to have unit Frobenius norm.
\figref{fig:ex-iTRc-vs-iTR} shows the Rayleigh quotient $\theta$ and its
difference with the exact solution as a function of the iteration count.
Note that the results for both variants of \algref{alg:fpi-can} are
indistinguishable, hence, we will only use the faster non-canonical variant
in the remainder of this section.

\begin{figure}[hbtp!]
\centering%
\subfloat[Comparison of the canonical and non-canonical variants of
          \algref{alg:fpi-can} for time step $t = 0.01$.%
          \label{fig:ex-iTRc-vs-iTR}]{%
\figname{ex-iTR-vs-iTRc-lam}%
\begin{tikzpicture}
\begin{axis}[
  width=0.48\textwidth,
  xmin=0, xmax=100,
  ymin=-2.3, ymax=-0.7,
  xlabel={iteration},
  ylabel={$\theta$},
  legend pos=north east,%
  legend style={draw=none,fill=none},%
]
\addplot[black,densely dashed] coordinates {(  0,-2.127088819946730)
                                            (100,-2.127088819946730)};
\addplot[myColOne,only marks,mark=x] table[x index=0,y index=3]%
 {\datfile{ex-ising2-can-r10-tau=1e-02}};
\addplot[myColTwo,only marks,mark=+] table[x index=0,y index=3]%
 {\datfile{ex-ising2-std-r10-tau=1e-02}};
\legend{,iTR2c,iTR2}
\end{axis}
\end{tikzpicture}%
\hfill%
\figname{ex-iTR-vs-iTRc-res}%
\begin{tikzpicture}
\begin{semilogyaxis}[
  width=0.48\textwidth,
  xmin=0, xmax=100,
  ymin=3e-5, ymax=3e0,
  xlabel={iteration},
  ylabel={$\abs{\lambda_0 - \theta}$},
  legend pos=north east,%
  legend style={draw=none,fill=none},%
]
\addplot[myColOne,only marks,mark=+] table[x index=0,y index=5]%
 {\datfile{ex-ising2-can-r10-tau=1e-02}};
\addplot[myColTwo,only marks,mark=x] table[x index=0,y index=5]%
 {\datfile{ex-ising2-std-r10-tau=1e-02}};
\legend{iTR2c,iTR2}
\end{semilogyaxis}
\end{tikzpicture}%
}\\%
\subfloat[Influence of the time step $t$ on the convergence of
          \algref{alg:fpi-can}.\label{fig:ex-fixed}]{%
\figname{ex-fixed}%
\begin{tikzpicture}
\begin{semilogyaxis}[
  width=0.98\textwidth,
  height=0.41\textwidth,
  xmin=0, xmax=5000,
  ymin=1e-8, ymax=1e0,
  xtick={0,500,1000,2000,...,5000},
  extra x ticks={100,200,300,400},
  extra x tick labels={},
  ytick={1e-8,1e-6,1e-4,1e-2,1e0},
  xlabel={iteration},
  ylabel={$\abs{\lambda_0 - \theta}$},
  legend pos=north east,%
  legend style={draw=none,fill=none},%
]
\addplot[myColOne,very thick] table[x index=1,y index=5]%
 {\datfile{ex-ising2-fixed1e3-r10-tau=1e-01}};
\addplot[myColTwo,very thick] table[x index=1,y index=5]%
 {\datfile{ex-ising2-fixed1e3-r10-tau=1e-02}};
\addplot[myColThr,very thick] table[x index=1,y index=5]%
 {\datfile{ex-ising2-fixed1e3-r10-tau=1e-03}};
\addplot[myColFou,very thick] table[x index=1,y index=5]%
 {\datfile{ex-ising2-fixed1e3-r10-tau=1e-04}};
\addplot[myColFiv,very thick] table[x index=1,y index=5]%
 {\datfile{ex-ising2-fixed1e3-r10-tau=1e-05}};
\addplot[myColOne,very thick] table[x index=1,y index=5]%
 {\datfile{ex-ising2-fixed1e2-r10-tau=1e-01}};
\addplot[myColTwo,very thick] table[x index=1,y index=5]%
 {\datfile{ex-ising2-fixed1e2-r10-tau=1e-02}};
\addplot[myColThr,very thick] table[x index=1,y index=5]%
 {\datfile{ex-ising2-fixed1e2-r10-tau=1e-03}};
\addplot[myColFou,very thick] table[x index=1,y index=5]%
 {\datfile{ex-ising2-fixed1e2-r10-tau=1e-04}};
\addplot[myColFiv,very thick] table[x index=1,y index=5]%
 {\datfile{ex-ising2-fixed1e2-r10-tau=1e-05}};
\legend{$t = 1$e$-1$,$t = 1$e$-2$,$t = 1$e$-3$,$t = 1$e$-4$,$t = 1$e$-5$}
\end{semilogyaxis}
\end{tikzpicture}%
}\\%
\subfloat[Convergence as a function of iteration count $i$ and total time $T = \sum_i t_i$.%
\label{fig:ex-ising2-tau}]{%
\figname{ex-ising2-t-iter}\myplotiter[5]{ex-ising2-r10}%
{
  xmin=1e0, xmax=1e6,
  ymin=1e-12, ymax=1e0,
  xlabel={iteration},
  ylabel={$\abs{\lambda_0 - \theta}$},
  xtick={1e0,1e2,1e4,1e6},
  extra x ticks={1e1,1e3,1e5},
  extra x tick labels={},
  ytick={1e-12,1e-8,1e-4,1e0},
  extra y ticks={1e-10,1e-6,1e-2},
  extra y tick labels={}
}%
\hspace{0.01\textwidth}%
\figname{ex-ising2-t-time}\myplottime[5]{ex-ising2-r10}%
{
  xmin=0, xmax=8,
  ymin=1e-12, ymax=1e0,
  xlabel={time $T$},
  ylabel={$\abs{\lambda_0 - \theta}$},
  xtick={0,2,...,8},
  extra x ticks={1,3,...,7},
  extra x tick labels={},
  ytick={1e-12,1e-8,1e-4,1e0},
  extra y ticks={1e-10,1e-6,1e-2},
  extra y tick labels={},
}%
}\\%
\caption{One-dimensional transverse field Ising model with $g = 2$
and iTR rank $r = 10$.\label{fig:ex-ising2}}
\end{figure}
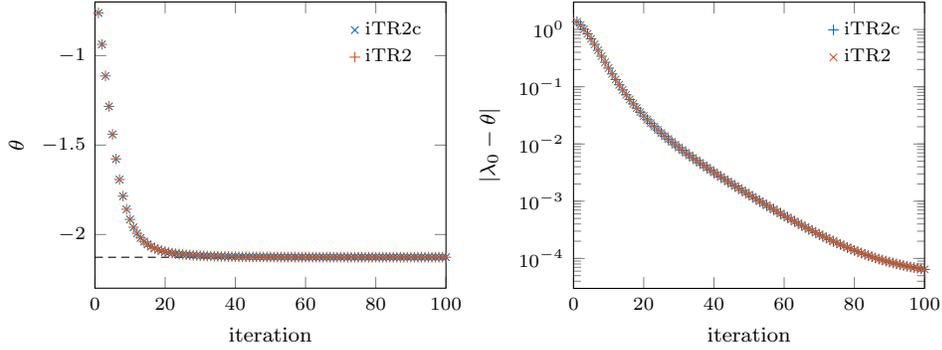
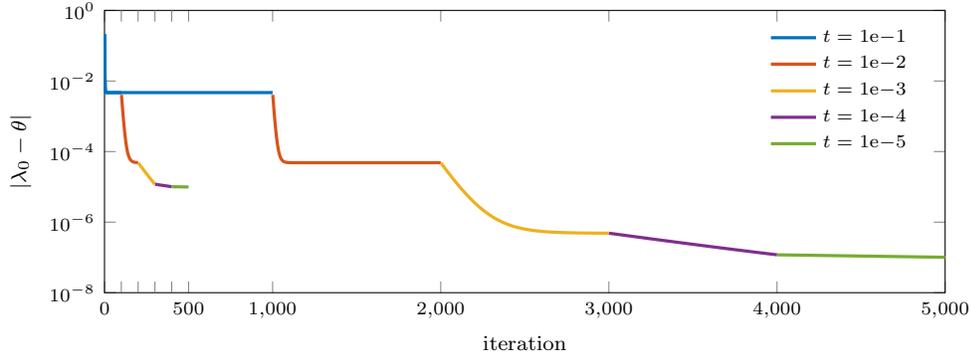
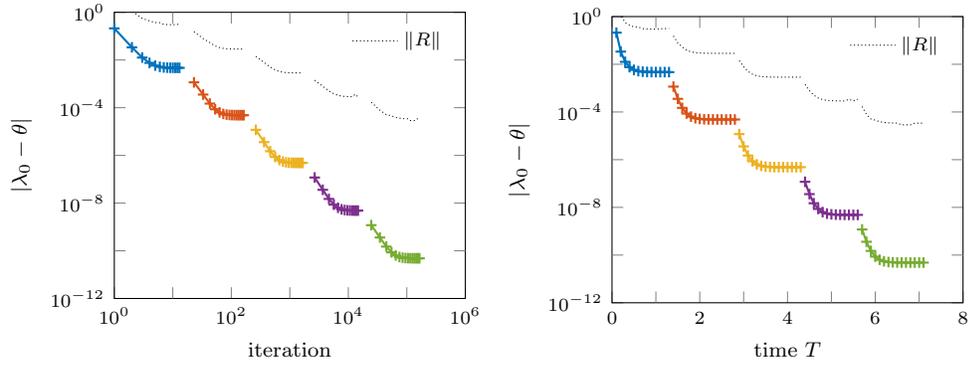

Next, we analyze the effect of the time step $t$ on the convergence
of \algref{alg:fpi-can}.
\figref{fig:ex-fixed} depicts the convergence of 2 different runs where we
respectively decreased the time step $t$ after every 100 or 1000 iterations.
Note that different time steps are indicated by different colors.
From this figure it is clear that the time step control
can have a significant effect on the overall attained accuracy.
When $t$ is decreased too soon, as in the first run, the power method
tend to stagnate and further decreasing $t$ does not help.
On the other hand, we do not want to run too many iteration and therefore
we will use the iTR residual, introduced in \secref{sec:res},
to detect stagnation.

\begin{table}[b!]
\label{tab:ising}\centering\footnotesize%
\caption{Wall clock times for one-dimensional transverse field Ising model
         with $g = 2$ and iTR rank $r = 10$.}
\begin{tabularx}{0.6\columnwidth}{C%
|S[table-format=7]S[table-format=3.2]%
|S[table-format=7]S[table-format=3.2]}
\toprule
$t$ & \multicolumn{1}{c}{iter} & \multicolumn{1}{c|}{time [s]}
    & \multicolumn{1}{c}{iter} & \multicolumn{1}{c}{time [s]} \\
\midrule
1e$-$1  &          &           &      13  &     0.87 \\
1e$-$2  &          &           &     150  &     0.72 \\
1e$-$3  &          &           &    1500  &     1.40 \\
1e$-$4  &          &           &   13000  &     4.80 \\
1e$-$5  &  370000  &   122.52  &  150000  &    44.87 \\
\midrule
Total   &  370000  &   122.52  &  164663  &    52.67 \\
\bottomrule
\end{tabularx}
\end{table}

As discussed in \secref{sec:alg}, the total number of iterations can be
reduced by gradually decreasing the time step $t$.
\tabref{tab:ising} shows that running \algref{alg:fpi-can} with a fixed
time step $t = 1$e$-5$ requires at least twice the number of iterations and
wall clock time than the adaptive time step approach, where $t$ was divided
by 10 when the first 3 digits of the residual norm $\norm\res$ were not changing
anymore in 3 consecutive convergence checks or in case the residual increased.
The left plot in \figref{fig:ex-ising2-tau} shows the corresponding convergence
as a function of the iteration count.
Note that by using the automatic time step adaption approach, the accuracy
of the Ritz value is multiple orders of magnitude higher than achieved in
both runs shown in \figref{fig:ex-fixed}.
The right plot in \figref{fig:ex-ising2-tau} shows the same results as
the left one but in function of the total time $T = N t$, with $N$ the
number of iterations of \algref{alg:fpi-can} for each $t$.
Remark that the stagnation happens for every time step more or less at the same
total time $T$, corresponding to roughly speaking the same power of $e^{-\bH}$.
We will use this observation in the next numerical experiments to change the
frequency of the convergence check based on the current time step.

\subsection{Spin $S=1$ Heisenberg isotropic antiferromagnetic model}

As a second example, we consider the isotropic antiferromagnetic case of the
spin $S = 1$ Heisenberg model.
The resulting Hamiltonian has the form of \eqref{eq:H}, with the following
nearest neighbor interaction matrix
\[
M_{k,k+1} = X_k \otimes X_{k+1} + Y_k \otimes Y_{k+1}
                         + \Delta Z_k \otimes Z_{k+1},
\]
where $\Delta = 1$, and $X_k$, $Y_k$, and $Z_k$ the spin-1 generators of SU(2)
\begin{align*}
X_k &= \frac{1}{\sqrt{2}}
\begin{bmatrix}
0 & 1 & 0\\
1 & 0 & 1\\
0 & 1 & 0
\end{bmatrix}, &
Y_k &= \frac{1}{\sqrt{2}}
\begin{bmatrix}
 0 & -\I &   0\\
\I &   0 & -\I\\
 0 &  \I &   0
\end{bmatrix}, &
Z_k &=
\begin{bmatrix}
1 & 0 &  0\\
0 & 0 &  0\\
0 & 0 & -1
\end{bmatrix}.
\end{align*}
The smallest eigenvalue is $\lambda_0 \approx -1.401\,484\,038\,971\,2(2)$
\cite{haci2011}.

\newcommand{\myplottimeafm}[2][5]{%
\myplottime[#1]{#2}{%
  xmin=0, xmax=45,%
  ymin=3e-9, ymax=1e0,%
  xlabel={time $T$},%
  ylabel={$\abs{\lambda_0 - \theta}$},%
  xtick={0,10,...,40}}%
}

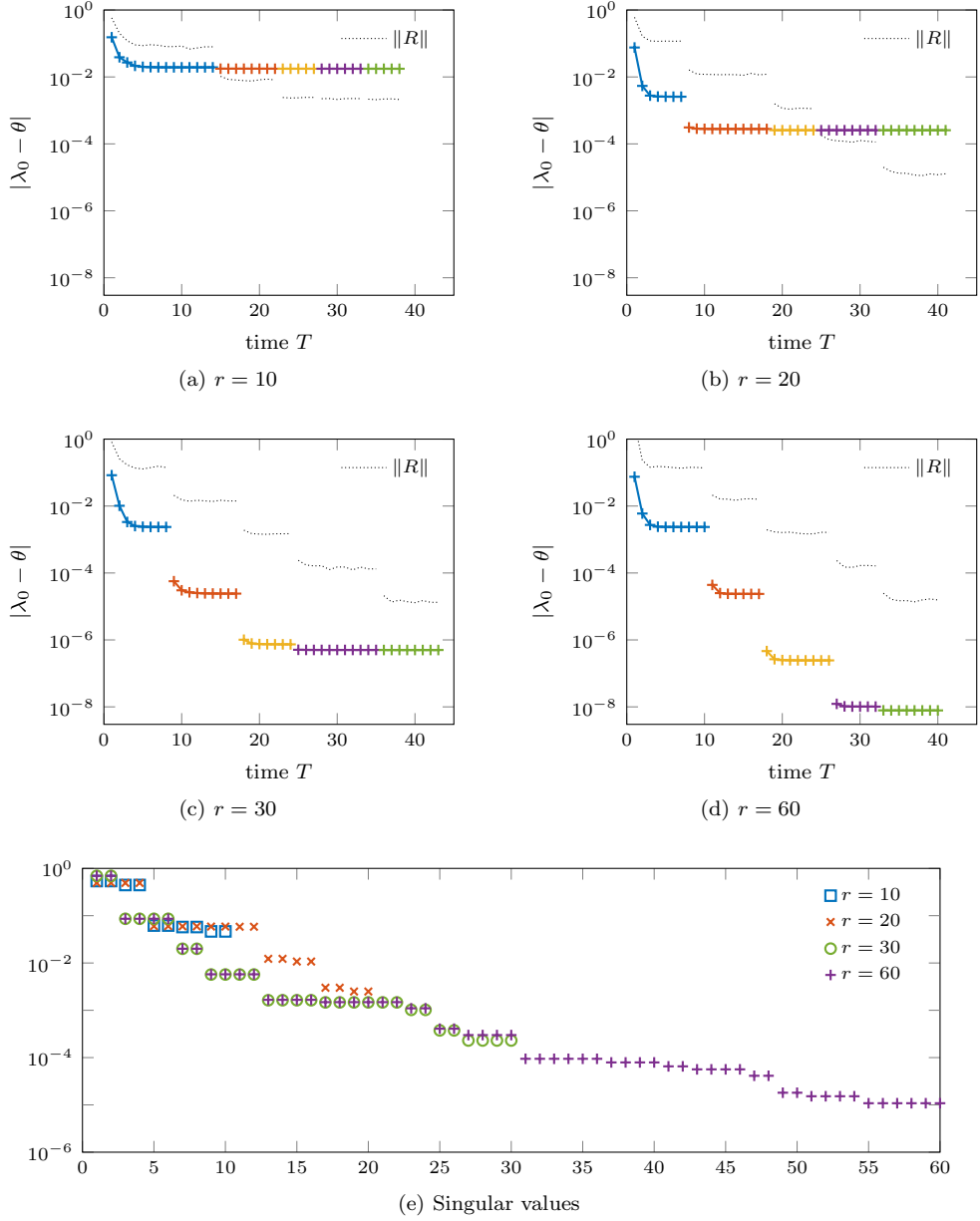
\begin{figure}[hbtp!]
\centering%
\subfloat[$r = 10$\label{fig:ex-afm-r10}]{%
\figname{ex-afm-res-r10}\myplottimeafm[5]{ex-afm-r10}%
}\hfill%
\subfloat[$r = 20$\label{fig:ex-afm-r20}]{%
\figname{ex-afm-res-r20}\myplottimeafm[5]{ex-afm-r20}%
}\\%
\subfloat[$r = 30$\label{fig:ex-afm-r30}]{%
\figname{ex-afm-res-r30}\myplottimeafm[5]{ex-afm-r30}%
}\hfill%
\subfloat[$r = 60$\label{fig:ex-afm-r60}]{%
\figname{ex-afm-res-r60}\myplottimeafm[5]{ex-afm-r60}%
}\\%
\subfloat[Singular values\label{fig:ex-afm-svd}]{%
\figname{ex-afm-svd}%
\begin{tikzpicture}
\begin{semilogyaxis}[
  width=\textwidth,
  height=0.4125\textwidth,
  xmin=0, xmax=60,
  ymin=1e-6, ymax=1e0,
  ytick={1e-6,1e-4,1e-2,1e0},
  extra y ticks={1e-5,1e-3,1e-1},
  extra y tick labels={},
  legend pos=north east,%
  legend style={draw=none,fill=none},%
]
\addplot[only marks,myColOne,thick,mark=square] table[x index=0,y index=1]%
 {\datfile{ex-afm-r10-svd}};
\addplot[only marks,myColTwo,thick,mark=x] table[x index=0,y index=1]%
 {\datfile{ex-afm-r20-svd}};
\addplot[only marks,myColFiv,thick,mark=o] table[x index=0,y index=1]%
 {\datfile{ex-afm-r30-svd}};
\addplot[only marks,myColFou,thick,mark=+] table[x index=0,y index=1]%
 {\datfile{ex-afm-r60-svd}};
\legend{$r = 10$,$r = 20$,$r = 30$,$r = 60$}
\end{semilogyaxis}
\end{tikzpicture}%
}%
\caption{Spin $S=1$ Heisenberg isotropic antiferromagnetic model.%
\label{fig:ex-afm}}
\end{figure}

\figref{fig:ex-afm-r10,fig:ex-afm-r20,fig:ex-afm-r30,fig:ex-afm-r60}
show the convergence of \algref{alg:fpi-can} for different iTR ranks
$r = 10$, 20, 30, and 60, respectively.
For every case, we started with a time step $t = 0.1$ and the time step is
divided by a factor of 10 whenever residual stagnation is detected with the
criteria stated in the previous example.
The frequency of the convergence check was set to $1/t$ in order to avoid
computing to many (expensive) residual calculations.
The corresponding wall clock time for the different runs are given in
\tabref{tab:afm}.
This table illustrates again that the smaller the time step $t$ the
more iterations are required, hence, the larger the wall clock time.

When the rank of the iTR is too small, we see from \figref{fig:ex-afm} that the
Rayleigh quotient does not further improve even when the time step is reduced
because the SVD truncation error is too large.
On the other hand, we observe from
\figref{fig:ex-afm-r20,fig:ex-afm-r30,fig:ex-afm-r60} that the residual
norms in all these runs behave in a similar fashion.
Since the iTR2 residual \eqref{eq:iTR2c-res-th} is a projected residual,
we cannot use it as the only metric for deciding when to terminate the
power iteration.
As illustrated in \figref{fig:ex-afm-svd}, the singular values given by
the diagonal matrices of \defref{def:iTR-can2} and in particular the
smallest singular value together with the residual are good indicators
for the actual accuracy of the approximate eigenpair of $\bH$ because
the overall accuracy is bounded by both the Suzuki--Trotter splitting
and SVD truncation error.

\begin{table}[htbp]
\label{tab:afm}\centering\footnotesize%
\caption{Wall clock times corresponding to
         \figref{fig:ex-afm-r10,fig:ex-afm-r20,fig:ex-afm-r30,fig:ex-afm-r60}.}
\begin{tabularx}{\columnwidth}{C|%
S[table-format=6]S[table-format=4.2]|%
S[table-format=6]S[table-format=4.2]|%
S[table-format=6]S[table-format=4.2]|%
S[table-format=6]S[table-format=4.2]}
\toprule
\multirow{2}{*}{$t$} & \multicolumn{2}{c|}{$r = 10$}
                     & \multicolumn{2}{c|}{$r = 20$}
                     & \multicolumn{2}{c|}{$r = 30$}
                     & \multicolumn{2}{c}{$r = 60$} \\
 & \multicolumn{1}{c}{iter} & \multicolumn{1}{c|}{time [s]}
 & \multicolumn{1}{c}{iter} & \multicolumn{1}{c|}{time [s]}
 & \multicolumn{1}{c}{iter} & \multicolumn{1}{c|}{time [s]}
 & \multicolumn{1}{c}{iter} & \multicolumn{1}{c}{time [s]} \\
\midrule
1e$-$1  &     140  &     0.60  &      70  &     0.71
        &      80  &     1.04  &     100  &     3.83 \\
1e$-$2  &     800  &     0.75  &    1100  &     2.26
        &     900  &     3.28  &     700  &     9.08 \\
1e$-$3  &    5000  &     3.27  &    6000  &     8.63
        &    7000  &    19.60  &    9000  &    96.78 \\
1e$-$4  &   60000  &    33.75  &   80000  &   107.57
        &  110000  &   314.33  &   60000  &   590.71 \\
1e$-$5  &  500000  &   288.10  &  900000  &  1190.97
        &  800000  &  2137.29  &  800000  &  7521.75 \\
\bottomrule
\end{tabularx}
\end{table}

\subsection{Spin $S=1/2$ Heisenberg model}

As a last example, we consider the spin $S = 1/2$ Heisenberg model.
The resulting Hamiltonian has the form of \eqref{eq:H}, with the following
nearest neighbor interaction matrix
\[
M_{k,k+1} = X_k \otimes X_{k+1} + Y_k \otimes Y_{k+1} + Z_k \otimes Z_{k+1},
\]
where $X_k = \frac{1}{2}\sigma_x$, $Y_k = \frac{1}{2}\sigma_y$,
$Z_k = \frac{1}{2}\sigma_z$, with $\sigma_x$, $\sigma_y$, $\sigma_z$
the Pauli matrices.
The exact solution of the smallest eigenvalue is
\(
\lambda_0 = -\ln(2) + \frac{1}{4} \approx -0.443\,147\,180\,559\,945
\).

In \figref{fig:ex-spin-theta12}, we plot the convergence history of
the Rayleigh quotient $\theta$ to the exact eigenvalue $\lambda_0$ as well
as the changes in the first ($\theta_1$) and second ($\theta_2$) terms of
\eqref{eq:rq2} (excluding the 1/2 factor) with respect to the total time $T$.
In this calculation, we limit the rank of the iTR to 10 and used a fixed
time step $t = 0.001$.
As we can see, $\theta_1$ approaches to $\lambda_0$ from above, but
$\theta_2$ can fall below $\lambda_0$ and eventually approaches $\lambda_0$
from below. Although $\theta_1$ and $\theta_2$ both converge towards
$\lambda_0$, the convergence is much slower than $\theta$ which is the
average of $\theta_1$ and $\theta_2$.

In addition to approximating the desired eigenvalue by the Rayleigh quotient,
we can obtain another type of approximation by computing the eigenvalue
of the projected Hamiltonian matrix $H_{\neq0}$ defined by \eqref{eq:Htsum}.
Because the summation in $H_{\neq0}$ diverges, the eigenvalue of
$H_{\neq0}$ is not finite. However, the eigenvalue of the average of
$H_{\neq0}$ is finite and serves as an approximation to the average eigenvalue
of $\bH$ defined in \eqref{eq:H}.
To compute the average of $H_{\neq0}$, we can first subtract $\theta\eye$
from each of the terms $H_{-2-\ell}$ and $H_{1+\ell}$ in \eqref{eq:Ht},
for $\ell\geq 1$, where $\theta$ is the Rayleigh quotient.
We then compute the geometric sums
$H_L = H_{-2} + \sum_{\ell=1}^{\infty} (H_{-2-\ell}-\theta\eye)$ and
$H_R = H_1 + \sum_{\ell=1}^{\infty} (H_{1+\ell}-\theta\eye)$, respectively,
using the same techniques as discussed in the proof of \thref{th:iTRc-res}
for the average residual calculation.
We can show that both geometric sums converge.
Next, we sum up $H_L$, $H_{-1}$, $H_0$, and $H_R$ after $\theta\eye$ is
subtracted from each of these terms.
The subtraction of $\theta\eye$ from $H_L$ and $H_R$ are made to account for
the $\theta\eye$'s not subtracted from $H_{-2}$ and $H_{1}$, respectively,
before the geometric sums were performed to yield $H_L$ and $H_R$.
After dividing the sum of all four terms by 4, we add $\theta\eye$ back
to form the total average of $H_{\neq0}$, i.e.,
\[
\Hh_{\neq0} = (H_L + H_{-1} + H_{0} + H_R - 4 \theta\eye)/4 + \theta\eye
            = (H_L + H_{-1} + H_{0} + H_R)/4.
\]
In \figref{fig:ex-spin-theta12}, the curve marked by the diamonds show that
the eigenvalue $\thetah$ of the average $\Hh_{\neq0}$ converges faster to
$\lambda_0$.
However, computing the average $\Hh_{\neq0}$ can be expensive when the rank
of $\bfx$ becomes large.

\Figref{fig:ex-spin-r120} shows how the norm of the average residual and
the difference between the approximate and exact eigenvalue change with respect
to the total time $T$ when the rank of the iTR is allowed to increased to 120.
Even with this much larger rank, it is difficult to reduce the error in the
approximate eigenvalue below $10^{-6}$.

\begin{figure}[hbtp!]
\centering%
\subfloat[Convergence of the iTR Rayleigh quotient and its 2 comprising terms
          ($r = 10$ and $t = 0.001$).%
          \label{fig:ex-spin-theta12}]{%
\figname{ex-spin-theta12}%
\begin{tikzpicture}
\begin{axis}[
  width=\textwidth,
  height=0.5\textwidth,
  xmin=0, xmax=10,
  ymin=-0.5, ymax=0,
  ytick={-0.5,-0.4,-0.3,-0.2,-0.1,0},
  extra y ticks={-0.4431},
  extra y tick labels={$\lambda_0$},
  xlabel={time $T$},
  ylabel={$\theta$},
  legend pos=north east,%
  legend style={draw=none,fill=none},%
]
\addplot[black,densely dashed] coordinates {( 0,-0.4431471805599453)
                                            (10,-0.4431471805599453)};
\addplot[myColOne,only marks,mark=asterisk] table[x index=2,y index=3]%
 {\datfile{ex-spin-theta12-r10-tau=1e-03}};
\addplot[myColTwo,only marks,mark=x] table[x index=2,y index=4]%
 {\datfile{ex-spin-theta12-r10-tau=1e-03}};
\addplot[myColThr,only marks,mark=+] table[x index=2,y index=5]%
 {\datfile{ex-spin-theta12-r10-tau=1e-03}};
\addplot[myColFou,only marks,mark=diamond] table[x index=2,y index=6]%
 {\datfile{ex-spin-theta12-r10-tau=1e-03}};
\legend{,$\theta$,$\theta_1$,$\theta_2$,$\thetah$}
\end{axis}
\end{tikzpicture}%
}\\%
\subfloat[Convergence as a function of total time for $r = 120$.%
          \label{fig:ex-spin-r120}]{%
\figname{ex-spin-res-r120}\myplottime[4]{ex-spin-r120}%
{
  width=0.99\textwidth,
  height=0.5\textwidth,
  xmin=0, xmax=600,
  ymin=1e-7, ymax=1e0,
  xlabel={time $T$},
  ylabel={$\abs{\lambda_0 - \theta}$},
  xtick={0,100,...,600},
  extra x ticks={50,150,...,550},
  extra x tick labels={},
  ytick={1e-6,1e-4,1e-2,1e0},
  extra y ticks={1e-7,1e-5,1e-3,1e-1},
  extra y tick labels={},
}%
}%
\caption{Spin $S=1/2$ Heisenberg model.\label{fig:ex-spin}}
\end{figure}
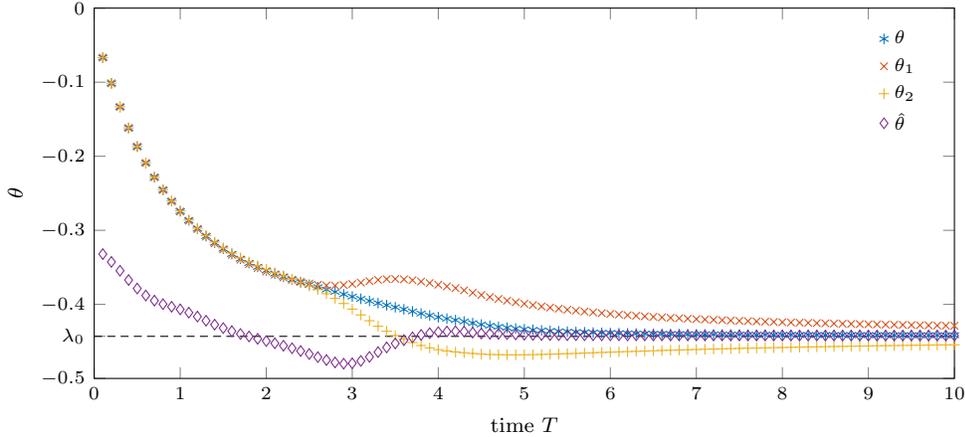
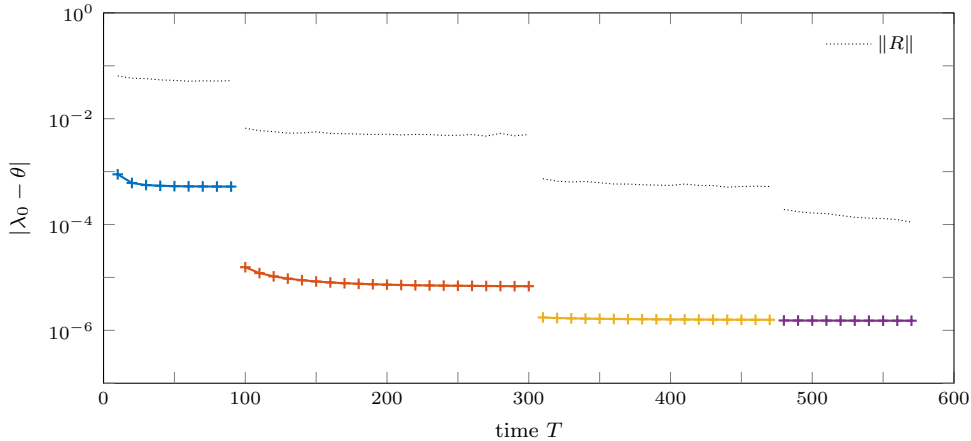

\section{Conclusions}
\label{sec:concl}

In this paper, we described a power method for solving
a class of infinite-dimensional tensor eigenvalue problems for an
infinite-dimensional matrix $\bH$ with
a special type of Kronecker product structure, assuming
that the approximate eigenvector can be represented by an infinite
translational invariant tensor ring (iTR).
We presented some algebraic properties of the iTR in terms of its transfer
matrix and showed how to obtain a canonical form, which is convenient for
manipulating an iTR in several calculations.
We showed how the Rayleigh-quotient per site of an iTR $\bfx$ with respect to
the matrix $\bH$ can be efficiently computed.
Although the power method itself is not new,
the purpose of our paper is to clearly formulate this type of problem from
a numerical linear algebra point of view and describe how a standard numerical
algorithm such as the power method can be modified and applied to solve
this type of problem. We have provided algorithmic
and implementation details on how to carry out this procedure
in a practical setting. In particular, we presented a better way to compute
the Rayleigh quotient per site, and pointed out an alternative way to extract possibly
better approximations to the desired eigenvalue. We also developed a practical
and effective way to monitor the convergence of the power iteration
by computing a projected residual estimate. Such a residual estimate
is used to determine when to decrease the parameter $t$ to improve convergence and when
to terminate the power iteration.

We should note that the flexible power iteration is not the only method
for solving this type of infinite-dimensional tensor eigenvalue problems.
Alternative methods include the infinite density matrix renormalization (iDMRG)
method and the variational uniform matrix product states (vUMPS) method.
Due to the infinite dimensionality of the eigenvalue
problem, special care must be exercised to compute quantities such as the
Rayleigh quotient and the residual of an approximate eigenpair.
We also need to take into account the approximation that must be made in the
multiplication of the matrix exponential $e^{-\bH t}$ with an iTR.
A comparison to other algorithms and higher order
Lie--Trotter product formulas will be made in a future study.

\bibliographystyle{siamplain}
\bibliography{references}

\end{document}